\title{Spectral analysis of the biharmonic operator
subject to Neumann boundary conditions on dumbbell domains}
\author{Jos\'{e} M. Arrieta\thanks{Dep. Matem\'atica Aplicada, Universidad Complutense de Madrid, 28040 Madrid, Spain and Instituto de Ciencias Matem\'aticas CSIC-UAM-UC3M-UCM, \footnotesize{\texttt{arrieta@mat.ucm.es}}}
        \and
        Francesco Ferraresso\thanks{Dipartimento di matematica ``Tullio Levi-Civita'', Universit\`a degli Studi di Padova, 35121 Padova, Italy, \texttt{fferrare@math.unipd.it}}
        \and
        Pier Domenico Lamberti\thanks{Dipartimento di matematica ``Tullio Levi-Civita'', Universit\`a degli Studi di Padova, 35121 Padova, Italy, \texttt{lamberti@math.unipd.it}}
        }
        
\documentclass[a4paper,12pt]{article}

\usepackage{geometry}
\usepackage[english]{babel}
\usepackage{amsmath}
\usepackage{amssymb, amsfonts}
\usepackage{amsthm}
\usepackage{libertine}
\usepackage[libertine]{newtxmath}
\usepackage[T1]{fontenc}
\usepackage[utf8]{inputenc}
\usepackage{mathtools}
\usepackage{graphicx}
\usepackage{enumitem}
\usepackage[babel]{microtype}
\usepackage{hyperref}
\usepackage{mathrsfs}
\usepackage{courier}

\DeclarePairedDelimiter{\abs}{\lvert}{\rvert}
\DeclarePairedDelimiter{\norma}{\lVert}{\rVert}

\DeclareMathOperator{\Div}{div}

\newcommand{\numberset}{\mathbb}
\newcommand{\N}{\numberset{N}}

\newcommand{\R}{\numberset{R}}
\newcommand{\D}{\mathscr{D}}
\newcommand{\E}{\mathcal{E}}

\newcommand{\Hi}{\mathcal{H}}
\newcommand{\LL}{\mathcal{L}}

\newcommand{\eps}{\epsilon}
\newcommand{\M}{\mathcal{M}}

\newtheorem{theorem}{Theorem}
\newtheorem{lemma}{Lemma}
\newtheorem{proposition}{Proposition}
\newtheorem{definition}{Definition}
\theoremstyle{plain}

\newtheorem{remark}{Remark}

\numberwithin{equation}{section}

\begin{document}

\maketitle

\begin{abstract}
We consider the biharmonic operator subject to homogeneous boundary conditions of Neumann type on a
planar dumbbell domain which consists of two disjoint domains connected by a thin channel.
We analyse the spectral behaviour of the operator,  characterizing the limit of the eigenvalues and of the eigenprojections as the thickness of the channel goes to zero.  In applications to linear elasticity, the fourth order operator under consideration  is related to  the deformation
of a  free elastic plate, a part of which shrinks to a  segment. In contrast to what happens with the classical second order case, it turns out that the limiting equation is here distorted  by a strange factor depending on a parameter which plays the role of the Poisson coefficient of the represented plate.
\end{abstract}

%%% ----------------------------------------------------------------------
%\maketitle
%%% ----------------------------------------------------------------------
%\tableofcontents

\section{Introduction}
This paper is devoted to a spectral analysis of the biharmonic operator subject to   Neumann boundary conditions  on a domain which undergoes a singular perturbation.
The focus is on    planar dumbbell-shaped domains $\Omega_{\epsilon}$, with $\epsilon >0$, described in Figure~\ref{fig: dumbbell}. Namely,
given two   bounded  smooth domains  $\Omega_L, \Omega_R$ in $\R^2$ with $\Omega_L\cap \Omega_R=\emptyset $  such that $\partial \Omega_L \supset \{(0,y)\in \R^2 : -1<y<1\}$, $\partial \Omega_R  \supset \{(1,y)\in \R^2 : -1<y<1\}$, and $(\Omega_R\cup \Omega_L) \cap \left([0,1]\times[-1,1]\right) = \emptyset$, we
set
\[\Omega = \Omega _{L}\cup \Omega_R, \ \ {\rm and}\ \  \Omega_\eps = \Omega \cup R_\eps \cup L_\eps\, , \]
for all $\epsilon >0$ small enough. Here
$ R_\eps \cup L_\eps$ is a thin channel connecting $\Omega_L$ and $\Omega_R$  defined by
\begin{equation}
\label{def: R_eps}
R_\eps = \{(x,y)\in\R^2 : x\in(0,1), 0< y< \eps g(x) \},
\end{equation}
\[ L_\eps =( \{0\} \times (0, \eps g(0)) \cup (\{1\}\times (0, \eps g(1))) ), \]
where  $g \in C^2[0,1]$ is a positive real-valued function. Note that $\Omega_\eps$ collapses to the limit set $\Omega_0 = \Omega \cup ([0,1] \times \{0\})$ as $\eps \to 0$.

We consider  the eigenvalue problem
\begin{equation} \label{PDE: main problem_eigenvalues}
\begin{cases}
\Delta^2u - \tau \Delta u + u = \lambda \, u,   &\textup{in $\Omega_\eps$,}\\
(1-\sigma) \frac{\partial^2 u}{\partial n^2} + \sigma \Delta u = 0, &\textup{on $\partial \Omega_\eps$,}\\
\tau \frac{\partial u}{\partial n} - (1-\sigma)\, \Div_{\partial \Omega_\eps}(D^2u \cdot n)_{\partial \Omega_\eps} - \frac{\partial(\Delta u)}{\partial n} = 0, &\textup{on $\partial \Omega_\eps$,}
\end{cases}
\end{equation}
where  $\tau \geq 0$, $\sigma \in (-1,1)$ are fixed parameters,  and we analyse the behaviour of the eigenvalues  and of the corresponding eigenfunctions as $\eps \to 0$.  Here  $\Div_{\partial \Omega_\eps}$ is the tangential divergence operator, and $(\cdot)_{\partial \Omega_\eps}$ is the projection on the tangent line to $\partial \Omega_\eps$.
 The corresponding Poisson problem reads
\begin{equation} \label{PDE: main problem}
\begin{cases}
\Delta^2u - \tau \Delta u +u= f, &\textup{in $\Omega_\eps$},\\
(1-\sigma) \frac{\partial^2 u}{\partial n^2} + \sigma \Delta u = 0, &\textup{on $\partial \Omega_\eps$},\\
\tau \frac{\partial u}{\partial n} - (1-\sigma) \Div_{\partial \Omega_\eps}(D^2u \cdot n)_{\partial \Omega_\eps} - \frac{\partial(\Delta u)}{\partial n} = 0, &\textup{on $\partial \Omega_\eps$},
\end{cases}
\end{equation}
with  datum
$f \in L^2(\Omega_\eps)$.
\begin{figure}
\centering
\includegraphics[width=0.7\textwidth]{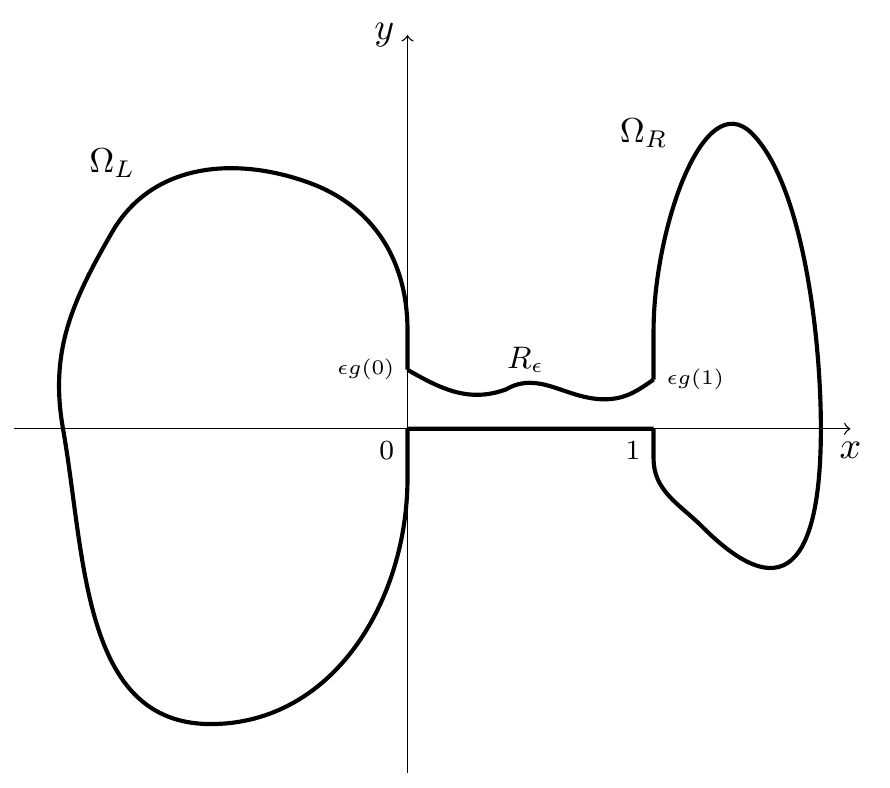}
\caption{The dumbbell domain $\Omega_\eps$.}
\label{fig: dumbbell}
\end{figure}

 Since $\partial\Omega_{\epsilon}$ has corner singularities at the junctions $(0,0)$,  $(0,\epsilon g(0))$, $(1,0)$, $(1,\epsilon g(1))$ and $H^{4}$
regularity does not hold around those points, we shall  always understand problems  \eqref{PDE: main problem_eigenvalues}, \eqref{PDE: main problem},
(as well as analogous problems) in a weak (variational) sense, in which case only $H^2$ regularity is required.

Namely, the variational formulation of problem \eqref{PDE: main problem} is the following: find $u\in H^2(\Omega_\eps)$ such that
\begin{equation} \label{PDE: main problem weak}
\int_{\Omega_\eps} (1-\sigma) D^2u : D^2\varphi + \sigma \Delta u \Delta \varphi + \tau \nabla u \cdot \nabla \varphi +u\varphi \,dx = \int_{\Omega_\eps} f \varphi\,dx\, ,
\end{equation}
for all $\varphi \in H^2(\Omega_\eps)$. The quadratic form associated with the left-hand side    of (\ref{PDE: main problem weak})  -  call it $B_{\Omega_{\epsilon}}(u, \varphi )$ -    is coercive for all $\tau \geq 0$  and  $\sigma \in (-1,1)$, see e.g. \cite{ChAppl}, \cite{Ch}.
In particular, by standard spectral theory  this quadratic form allows to define a non-negative self-adjoint operator $T=(\Delta^2 - \tau \Delta +I)_{N(\sigma )}$  in $L^2(\Omega_{\eps})$ which plays the role of   the classical operator $\Delta^2 - \tau \Delta +I$   subject to the  boundary conditions  above.
More precisely, $T$ is uniquely defined by the relation
$$ B_{\Omega_{\epsilon}}(u, \varphi )=<T^{1/2}u,T^{1/2}\varphi >_{L^2(\Omega_{\epsilon})} \, , $$
for all
$ u,\varphi \in H^2(\Omega_{\epsilon})$. In particular the domain of the square root $T^{1/2}$ of $T$ is $H^2(\Omega_{\epsilon})$ and
 a function $u$ belongs to the domain of $T$ if and only if
$u\in H^2(\Omega_{\epsilon})$
and there exists $f\in L^2(\Omega_{\epsilon})$ such that
$B_{\Omega_{\epsilon}}(u, \varphi )=  <f,\varphi >_{L^2(\Omega_{\epsilon})} $ for all $\varphi \in H^2(\Omega_{\epsilon})$, in which case
$Tu=f$. We refer to \cite[Chp.~4]{Daviesbook} for a general introduction to the variational approach  to the spectral analysis of partial  differential operators on non-smooth domains.

 The operator $T$  is densely defined  and its eigenvalues and eigenfunctions are exactly those of  problem (\ref{PDE: main problem_eigenvalues}).
Moreover,  since the embedding $H^2(\Omega_{\epsilon} )\subset L^2(\Omega_{\epsilon}  )$ is compact,   $(\Delta^2 - \tau \Delta +I)_{N(\sigma )}$ has compact resolvent, hence the spectrum is discrete and consists of a divergent increasing sequence of positive eigenvalues
$\lambda_{n}(\Omega_{\epsilon}),\  n\in \N$, with finite multiplicity (here each eigenvalue is repeated as many times as its multiplicity).

Problem (\ref{PDE: main problem_eigenvalues}) arises in linear elasticity in connection with the Kirchhoff-Love model for  the study of vibrations and deformations of free plates, in which case  $\sigma $ represents the
 Poisson ratio of the material and  $\tau$ the lateral tension. In this sense, the dumbbell domain $\Omega_{\epsilon}$  could  represent a plate and $R_{\epsilon }$
  a part of it which  degenerates to the   segment   $[0,1] \times \{0\}$.

We note that problem (\ref{PDE: main problem_eigenvalues}) can be considered as a natural fourth order version of the corresponding  eigenvalue problem for the
Neumann Laplacian $-\Delta_N$, namely
\begin{equation} \label{PDE: second problem_eigenvalues}
\begin{cases}
-\Delta  u  + u = \lambda \, u,   &\textup{in $\Omega_\eps$,}\\
\frac{\partial u}{\partial n} = 0, &\textup{on $\partial \Omega_\eps$,}\\
\end{cases}
\end{equation}
the variational formulation of which reads

\begin{equation} \label{PDE: second problem weak}
\int_{\Omega_\eps}  Du \cdot D \varphi +  u  \varphi \,dx =\lambda   \int_{\Omega_\eps} u   \varphi\,dx ,
\end{equation}
where the test functions $\varphi$ and the unknown $u$ are considered in $H^1(\Omega_{\eps})$.

Although the terminology used in the literature  to refer to boundary value problems for fourth order operators is sometimes a bit misleading, we emphasise
that the formulation  of  problems  \eqref{PDE: main problem_eigenvalues}, \eqref{PDE: main problem} is rather classical, see e.g.   \cite[Example~2.15]{necas}
where problem  \eqref{PDE: main problem} with $\tau =0$ is referred to as the Neumann problem for the biharmonic operator. Moreover, we point out that a number of  recent papers devoted to the  analysis of \eqref{PDE: main problem_eigenvalues} have confirmed that problem  (\ref{PDE: main problem_eigenvalues})
can be considered as   the natural Neumann problem for the biharmonic operator, see  \cite{arlacras}, \cite{ArrLamb}, \cite{BuosoProv}, \cite{BuosoProvCH}, \cite{bulacompl}, \cite{ChAppl}, \cite{Ch}, \cite{Prov}.
We also refer to  \cite{GazzGS}  for an extensive discussion on boundary value problems for higher order elliptic operators.

It is  known that  the eigenelements of the Neumann Laplacian   on a typical dumbbell domain as above have a singular behaviour, see \cite{ArrPhD}, \cite{Arr1}, \cite{Arr2}, \cite{ACJdE}, \cite{ACL}, \cite{AHH}, and the references therein.  For example,   it is  known that  not all the eigenvalues of  $-\Delta_N$  on $\Omega_{\epsilon}$ converge to  the eigenvalues of $-\Delta_N$ in $\Omega$; indeed, some of the eigenvalues of the dumbbell domain are asymptotically close to the eigenvalues of a boundary value problem  defined  in the channel $R_\eps$. This allows the appearance in the limit of extra eigenvalues associated with an ordinary differential equation in the segment $(0,1)$, which are  generally  different from the eigenvalues of $-\Delta_N$ in $\Omega$.
Such singular behaviour reflects a general characteristic of boundary value problems with Neumann boundary conditions, the stability of which requires rather strong assumptions on the admissible  domain perturbations, see e.g., \cite{ACJdE}, \cite{ArrLamb}, \cite{lalaneu}. We refer to \cite[p.~420]{C-H} for a classical counterexample.

The aim of the present paper is to clarify how Neumann boundary conditions affect the spectral behaviour of the operator $\Delta^2-\tau \Delta $ on dumbbell domains, by extending the validity of some results   known for the second order operator $-\Delta_N$ to the fourth-order operator $(\Delta^2 - \tau \Delta)_{N(\sigma )}$.

First of all, we  prove that the eigenvalues of problem (\ref{PDE: main problem_eigenvalues})
can be asymptotically decomposed into two families of eigenvalues as
\begin{equation}\label{dec}
(\lambda_n(\Omega_\eps))_{n\geq 1} \approx (\omega_k)_{k\geq 1} \cup (\theta^\eps_l)_{l\geq 1}, \ \ {\rm as }\  \eps \to 0,
\end{equation}
 where $(\omega_k)_{k\geq 1}$ are the eigenvalues  of problem
\begin{equation} \label{PDE: Omega}
\begin{cases}
\Delta^2 w - \tau \Delta w + w = \omega_k\, w, &\text{in $\Omega$},\\
(1-\sigma) \frac{\partial^2 w}{\partial n^2} + \sigma \Delta w = 0, &\textup{on $\partial \Omega$},\\
\tau \frac{\partial w}{\partial n} - (1-\sigma) \Div_{\partial \Omega}(D^2w \cdot n)_{\partial \Omega} - \frac{\partial(\Delta w)}{\partial n} = 0, &\textup{on $\partial \Omega$,}
\end{cases}
\end{equation}
and  $(\theta^\eps_l)_{l\geq 1}$ are the eigenvalues of problem
\begin{equation} \label{PDE: R_eps}
\begin{cases}
\Delta^2 v - \tau \Delta v + v = \theta^\eps_l\, v, &\text{in $R_\eps$},\\
(1-\sigma) \frac{\partial^2 v}{\partial n^2} + \sigma \Delta v = 0, &\textup{on $\Gamma_\eps$},\\
\tau \frac{\partial v}{\partial n} - (1-\sigma) \Div_{\Gamma_\eps}(D^2v \cdot n)_{\Gamma_\eps} - \frac{\partial(\Delta v)}{\partial n} = 0, &\textup{on $\Gamma_\eps$,}\\
v = 0 = \frac{\partial v}{\partial n}, &\text{on $L_\eps$.}
\end{cases}
\end{equation}
The decomposition \eqref{dec} is proved under the assumption that a certain condition on $R_\eps$, called H-Condition, is satisfied. We provide in particular a simple condition on the profile function $g$ which guarantees the validity of the H-Condition.

Thus, in order to analyse the behaviour of $\lambda_n(\Omega_\eps)$ as $\eps \to 0$, it suffices to study $\theta^\eps_l$ as $\epsilon \to 0$. To do so, we need to  pass to the limit in the variational formulation of problem \eqref{PDE: R_eps}. Since the domain $R_\eps$ collapses to a segment as $\eps \to 0$, we use thin domain techniques in order to find the appropriate limiting problem. As in the case of the Laplace operator, the limiting problem depends on the shape of the channel $R_\eps$ via the profile function $g(x)$. More precisely it can be written as follows
\begin{equation}\label{ODE: limit problem}
\begin{cases}
\frac{1 - \sigma^2}{g} (gh'')'' - \frac{\tau}{g}(gh')' + h = \theta h, &\text{in $(0,1)$}\\
h(0)=h(1)=0,&\\
h'(0)=h'(1)=0.&
\end{cases}
\end{equation}
This allows to  prove convergence results for the eigenvalues and eigenfunctions of problem \eqref{PDE: main problem}. The precise statement can be found in Theorem~\ref{lastthm}.
Roughly speaking, Theorem~\ref{lastthm} establishes the following alternative:
\begin{itemize} \item[(A)] either $\lambda_n(\Omega_\eps) \to \omega_k$, for some $k\geq 1$  in which case  the corresponding eigenfunctions  converge in $\Omega$ to the eigenfunctions associated with   $\omega_k$.
\item[(B)] or $\lambda_n(\Omega_\eps) \to \theta_l$ as $\eps \to 0$ for some $l\in \N$  in which case the corresponding eigenfunctions  behave in $R_\eps$ like the eigenfunctions
associated with $ \theta_l$.
\end{itemize}
Moreover, all eigenvalues $\omega_k$ and $\theta_l$ are reached in the limit by the eigenvalues $\lambda_n(\Omega_{\eps})$.

We find it  remarkable that for $\sigma\ne 0$ the limiting equation in (\ref{ODE: limit problem}) is distorted by the coefficient $1-\sigma^2\ne 1$. This phenomenon
shows that  the dumbbell problem for our fourth order problem  \eqref{PDE: main problem_eigenvalues} with $\sigma \ne 0$ is significantly different from the second order problem \eqref{PDE: second problem_eigenvalues} considered in the literature.

We also note that the Dirichlet problem for the operator $\Delta^2u - \tau \Delta u + u$, namely
\begin{equation} \label{PDE: dir}
\begin{cases}
\Delta^2u - \tau \Delta u + u = \lambda \, u,   &\textup{in $\Omega_\eps$,}\\
 u = 0, &\textup{on $\partial \Omega_\eps$,}\\
 \frac{\partial u}{\partial n} = 0, &\textup{on $\partial \Omega_\eps$}
\end{cases}
\end{equation}
is stable in the sense that its  eigenelements converge to those of the operator $\Delta^2- \tau \Delta + I$ in $\Omega$ as $\eps\to 0$. In other words,  as for the Laplace operator, in the case of Dirichlet boundary conditions, no eigenvalues from the channel $R_{\epsilon}$ appear in the limit as $\epsilon \to 0$. In fact, it is well known that Dirichlet eigenvalues on thin domains diverge to $+\infty$ as $\eps \to 0$, because of the Poincar\'e inequality.

In order to prove our results, we study  the convergence of the resolvent operators  $(\Delta^2 - \tau \Delta +I)_{N(\sigma , \tau)}^{-1}$ and this is done by using the notion of  $\E$-convergence, which is  a useful tool  in the analysis  of boundary value problems defined on variable domains, see e.g., \cite{ACL}, \cite{arlacras}, \cite{ArrLamb}.

We point out that, although many papers in the literature have been devoted to the spectral analysis of second order operators with either Neumann or Dirichlet boundary conditions  on dumbbell domains,  see \cite{Arr1}, \cite{Arr2}, \cite{Jimbo1}, \cite{Jimbo2} and references therein, very little seems to be known about these problems for higher order operators. We refer to \cite{taylor} for a recent analysis of
the dumbbell problem in the case of elliptic systems subject to Dirichlet boundary conditions.

Finally, we observe that it would be interesting to provide precise rates of convergence  for the eigenvalues $\lambda_n(\Omega_{\epsilon})$ and the corresponding eigenfunctions as $\epsilon \to 0$ in the spirit of the asymptotic analysis performed e.g., in \cite{Arr2}, \cite{Gady1}, \cite{Gady2}, \cite{Gady3}, \cite{Gady4},  \cite{Jimbo1}, \cite{Jimbo2}  for second order operators. However, in case of higher order operators, this seems a challenging problem and is not addressed here.

The paper is organized as follows. In Section \ref{sec: decomposition} we prove the asymptotic decomposition \eqref{dec} of the eigenvalues $\lambda_n(\Omega_\eps)$. This is achieved in several steps. In Theorem \ref{thm: upper bound} we provide a suitable  upper bound for the eigenvalue $\lambda_n(\Omega_\eps)$.  Then,  in Definition~\ref{def: H condition} we introduce an assumption on the shape of the channel $R_\eps$, called H-Condition,  which is needed to prove a lower bound for $\lambda_n(\Omega_\eps)$ as $\eps \to 0$,  see Theorem~\ref{thm: lower bound}. Finally,   we collect the results of the section in Theorem \ref{thm: eigenvalues decomposition} to deduce a convergence result for the eigenvalues and the eigenfunctions of problem \eqref{PDE: main problem_eigenvalues} under the assumption that the H-Condition holds. In Section \ref{sec: proof H condition regular dumbbells} we  show that a wide class of  regular dumbbell domains satisfy the H-Condition. In Section \ref{sec: thin plates} we study the convergence of the solutions of problem \eqref{PDE: R_eps} as $\eps \to 0$, we  identify the limiting problem in $(0,1)$,  and we prove the spectral convergence of problem \eqref{PDE: R_eps} to problem \eqref{ODE: limit problem}.
Finally, in Section \ref{conclusionsec} we combine the results of the previous sections and prove Theorem~\ref{lastthm}.

%----------------------------------END OF INTRODUCTION----------------------------------------------------------------------------------

\section{Decomposition of the eigenvalues} \label{sec: decomposition}
The main goal of this section is to prove the decomposition of the eigenvalues of problem \eqref{PDE: main problem_eigenvalues} into the two families of eigenvalues coming from \eqref{PDE: Omega} and \eqref{PDE: R_eps}. First of all we note that, since $\Omega_{\epsilon} $, $\Omega $ and $R_{\epsilon }$ are sufficiently regular, by standard spectral theory for differential operators it follows that  the operators associated with the quadratic forms appearing in the weak formulation of problems \eqref{PDE: main problem_eigenvalues}, \eqref{PDE: Omega}, \eqref{PDE: R_eps} have compact resolvents.  Thus, the spectra of such problems are discrete and consist of positive eigenvalues of finite multiplicity. The eigenpairs of problems \eqref{PDE: main problem_eigenvalues}, \eqref{PDE: Omega}, \eqref{PDE: R_eps} will be denoted by $(\lambda_n(\Omega_\eps), \varphi_n^\eps)_{n \geq 1}$, $(\omega_n, \varphi_n^\Omega)_{n \geq 1}$,  $(\theta_n^\eps, \gamma_n^\eps)_{n\geq 1}$ respectively, where    the three families of eigenfunctions $\varphi_n^\eps$, $\varphi_n^\Omega$,  $\gamma_n^\eps$ are complete orthonormal bases of the spaces $L^2(\Omega_{\epsilon})$, $L^2(\Omega )$, $L^2(R_{\epsilon})$ respectively.
Moreover we set $(\lambda_n^\eps)_{n\geq 1} = (\omega_k)_{k \geq 1} \cup (\theta_l^\eps)_{l\geq 1}$, where it is understood that the eigenvalues are arranged in increasing order and repeated according to their multiplicity. In particular if $\omega_k = \theta_l^\eps$ for some $k,l \in \N$, then such an eigenvalue is repeated in the sequence $(\lambda_n^\eps)_{n \geq 1}$ as many times as the sum of the multiplicities of $\omega_k$ and $\theta_l^\eps$. Let us note explicitly that the order in the sequence $(\lambda_n^\eps)_{n\geq 1}$ depends on $\eps$. For each $\lambda_n^\eps$ we define the function  $\phi^\eps_n \in H^2(\Omega) \oplus H^2(R_\eps)$ in the following way:
\begin{equation}
\label{def: phi_n 1}
\phi^\eps_n = \begin{cases}
        \varphi_k^\Omega, &\text{in $\Omega$},\\
        0, &\text{in $R_\eps$},
        \end{cases}
\end{equation}
if $\lambda_n^\eps = \omega_k$, for some $k \in \N$; otherwise
\begin{equation}
\label{def: phi_n 2}
\phi^\eps_n=\begin{cases}
        0, &\text{in $\Omega$},\\
        \gamma_l^\eps, &\text{in $R_\eps$},
        \end{cases}
\end{equation}
if $\lambda_n^\eps = \theta_l^\eps$, for some $l \in \N$. We observe that in the case $\lambda_n^\eps= \omega_k = \theta_l^\eps$ for some $k,l \in \N$, with $\omega_k$ of multiplicity $m_1$ and $\theta_l^\eps$ of multiplicity $m_2$  we agree  to order the eigenvalues (and the corresponding functions $\phi^\eps_n$) by listing first the $m_1$ eigenvalues $\omega_k$, then the remaining $m_2$ eigenvalues $\theta_l^\eps$.

Note that $(\phi^\eps_i, \phi^\eps_j)_{L^2(\Omega_\eps)} = \delta_{ij}$ where  $\delta_{ij}$ is the Kronecker symbol, that is $\delta_{ij}=0$ for $i\ne j$ and $\delta_{ij}=1$ for $i=j$. Note also that  although $\phi_n^\eps$ defined by \eqref{def: phi_n 2} are in $H^2(\Omega_\eps)$ (due to the Dirichlet boundary condition imposed in $L_\eps$), the function $\phi_n^\eps$ defined by  \eqref{def: phi_n 1} do not lie in $H^2(\Omega_\eps)$.
To bypass this problem we define a sequence of functions in $H^2(\Omega_\eps)$  by setting
\[
\xi_n^\eps =
\begin{cases}
E\varphi_k^\Omega, &\text{if $\lambda_n^\eps = \omega_k$,}\\
\phi^\eps_n, &\text{if $\lambda_n^\eps = \theta_l^\eps$},
\end{cases}
\]
where $E$ is a linear continuous extension operator mapping $H^2(\Omega)$ to $H^2(\R^N)$. Then it is easy to verify that for fixed $i,j$, we have  $(\xi^\eps_i, \xi^\eps_j)_{L^2(\Omega_\eps)}=\delta_{ij}+o(1)$ as $\eps \to 0$. Then for fixed $n$ and for  $\eps$ small enough, $\xi_1^\eps,\ldots,\xi_n^\eps$  are linearly independent.

Now we  prove an upper bound for the eigenvalues $\lambda_n(\Omega_\eps)$.
\begin{theorem}[Upper bound] \label{thm: upper bound}
Let $n\geq 1$ be fixed.  The eigenvalues $\lambda_n^\eps$ are uniformly bounded in  $\eps$ and
\begin{equation} \label{eq: upper bound}
\lambda_n(\Omega_\eps) \leq \lambda_n^\eps + o(1), \quad \text{as $\eps \to 0$.}
\end{equation}
\end{theorem}
\begin{proof}

The fact that $\lambda_n^\eps$ remains bounded as $\eps \to 0$  is an easy consequence of the inequality \begin{equation} \label{eq: boundedness lambda_n^eps}
\lambda_n^\eps \leq \omega_n < \infty,
\end{equation}
which holds by definition of $\lambda_n^\eps$.
In the sequel we write $\perp$ to denote the orthogonality in $L^2$, and $[f_1, \dots, f_m]$ for the linear span of the functions $f_1, \dots, f_m$.

By the variational characterization of the eigenvalues $\lambda_n(\Omega_\eps)$ we have
\begin{multline} \label{eq: lambda_n(Omega_eps)var}
\lambda_n(\Omega_\eps) = \min \left\{ \frac{\displaystyle \int_{\Omega_\eps} (1-\sigma) |D^2\psi|^2 + \sigma |\Delta \psi|^2 + \tau |\nabla \psi|^2 + |\psi|^2 }{\displaystyle\int_{\Omega_\eps} |\psi|^2}  \right.\\
\left. : \text{$\psi \in H^2(\Omega_\eps)$, $\psi \not\equiv 0$ and $\psi \perp \varphi_1^\eps, \dots, \varphi_{n-1}^\eps$}   \right\}.
\end{multline}
Since the functions $\xi^\eps_1,\dots,\xi^\eps_n$ are linearly independent, by a dimension argument there exists $\xi^\eps \in [\xi^\eps_1,\dots,\xi^\eps_n]$ such that $\norma{\xi^\eps}_{L^2(\Omega_\eps)}=1$, and $\xi^\eps \perp \varphi_1^\eps, \dots, \varphi_{n-1}^\eps$.

We can write
%\begin{equation}\label{upperbound}
$ \xi^\eps = \sum_{i=1}^n \alpha_i \xi_i^\eps$,
%\end{equation}
for some $\alpha_1,\dots, \alpha_n \in \R$ depending on $\eps$ such that $\sum_{i=1}^n \alpha_i^2 = 1 + o(1)$ as $\eps \to 0$. By using $\xi^\eps$ as a test function in \eqref{eq: lambda_n(Omega_eps)var} we get

\begin{equation} \label{proof: computationsRQ}
\begin{split}
&\lambda_n(\Omega_\eps) \leq \int_{\Omega_\eps} (1-\sigma) |D^2\xi^\eps|^2 + \sigma |\Delta \xi^\eps|^2 + \tau |\nabla \xi^\eps|^2 +|\xi^\eps|^2\\
&= \sum_{i=1}^n \alpha_i^2 \biggl( \int_{\Omega_\eps} (1-\sigma) |D^2\xi_i^\eps|^2 + \sigma  |\Delta\xi_i^\eps|^2 + \tau  |\nabla \xi_i^\eps|^2 + |\xi_i^\eps|^2 \biggr) \\
&+ \sum_{i\neq j}\alpha_i\alpha_j \biggl( \int_{\Omega_\eps} (1-\sigma) (D^2\xi^\eps_i : D^2\xi_j^\eps) + \sigma \Delta \xi_i^\eps \Delta\xi_j^\eps + \tau \nabla\xi_i^\eps \cdot \nabla \xi_j^\eps +  \xi_i^\eps \xi_j^\eps  \biggr).
\end{split}
\end{equation}

By definition of $\xi_i^\eps$ and the absolute continuity of the Lebesgue integral, we have
{\small\[\int_{\Omega_\eps} (1-\sigma) |D^2\xi_i^\eps|^2 + \sigma  |\Delta\xi_i^\eps|^2 + \tau  |\nabla \xi_i^\eps|^2 + |\xi_i^\eps|^2=
\begin{cases}
\omega_k+o(1), & \hbox{if }\textup{$\exists\, k$ s.t. $\lambda_i^\eps=\omega_k$} ,\\
\theta_\eps^l, &\hbox{if }\textup{$\exists\, l$ s.t. $\lambda_i^\eps=\theta_\eps^l$},
\end{cases}
\]}
which implies that
$\int_{\Omega_\eps} (1-\sigma) |D^2\xi_i^\eps|^2 + \sigma  |\Delta\xi_i^\eps|^2 + \tau  |\nabla \xi_i^\eps|^2 + |\xi_i^\eps|^2\leq \lambda_n^\eps+o(1).$

Note that
\[
\begin{split}
&\sum_{i\neq j}\alpha_i\alpha_j \biggl( \int_{\Omega_\eps} (1-\sigma) (D^2\xi^\eps_i : D^2\xi_j^\eps) + \sigma \Delta \xi_i^\eps \Delta\xi_j^\eps + \tau \nabla\xi_i^\eps \cdot \nabla \xi_j^\eps  +  \xi_i^\eps \xi_j^\eps   \biggr)=o(1).%\\
\end{split}
\]

Hence,
$\lambda_n(\Omega_\eps)\leq \sum_{i=1}^n \alpha_i^2 ( \lambda_n^\eps+o(1))+o(1)\leq \lambda_n^\eps+o(1)$
which concludes the proof of \eqref{eq: upper bound}.
\end{proof}

\begin{remark}
%The previous proof works also if $\tau < 0$, and it works in $\R^N$ for any $N\geq2$.
Note  that the shape of the channel $R_\eps$ does not play any role in establishing the upper bound. The only fact needed is that the measure of $R_\eps$ tends to $0$ as $\eps \to 0$.
\end{remark}

In the sequel we shall provide a lower bound for the eigenvalues $\lambda_n(\Omega_\eps)$. Before doing so, let us introduce some notation.

\begin{definition}\label{definitionNorm}
Let $\sigma \in (-1,1)$, $\tau \geq 0$. We denote by $H^2_{L_\eps}(R_\eps)$ the space obtained as the closure in $H^2 (R_\eps)$ of $C^{\infty}(\overline{R_{\epsilon}})$ functions which vanish in a neighbourhood of  $L_\eps$.
Furthermore, for any Lipschitz bounded open set $U$  we define
\[
[f]_{H^2_{\sigma,\tau}(U)} = \bigl|(1-\sigma) \norma{D^2 f}_{L^2(U)}^2 + \sigma \norma{\Delta f}_{L^2(U)}^2 + \tau \norma{\nabla f}_{L^2(U)}^2 + \norma{f}_{L^2(U)}^2   \bigr|^{1/2}\, ,
\]
for all $f \in H^2(U)$.
\end{definition}

Note the functions $u$ in $ H^2_{L_\eps}(R_\eps) $ satisfy the conditions $u=0$ and $\nabla u =0$ on $L_{\epsilon}$ in the sense of traces.

\begin{proposition} \label{prop: convergence eigenprojections}
Let $n \in \N$ be such that the following two conditions are satisfied:
\begin{enumerate}[label=(\roman*)]
\item  For all $i=1,\dots,n$,
\begin{equation} \label{prop: lambda_i}
\abs{\lambda_i^\eps - \lambda_i(\Omega_\eps)}\to 0 \quad \quad \text{as $\eps \to 0$,}
\end{equation}
\item There exists $\delta>0$ such that
\begin{equation} \label{prop: lambda_n+1}
\lambda_n^\eps \leq \lambda_{n+1}(\Omega_\eps) - \delta
\end{equation}
for any $\epsilon >0$ small enough.
\end{enumerate}
Let  $P_n$  be the projector from $L^2(\Omega_\eps)$ onto the linear span $[\phi_1^\eps,\dots,\phi_n^\eps]$ defined by
\begin{equation}
P_n g = \sum_{i=1}^n (g, \phi_i^\eps)_{L^2(\Omega_\eps)} \phi_i^\eps\, ,
\end{equation}
for all $g\in L^2(\Omega_\eps)$, where $\phi_i^\eps$ is defined in \eqref{def: phi_n 1}, \eqref{def: phi_n 2}. Then
\begin{equation} \label{prop: thesis}
\norma{\varphi_i^\eps - P_n \varphi_i^\eps}_{H^2(\Omega) \oplus H^2(R_\eps)} \to 0,
\end{equation}
as $\eps \to 0$, for all $i=1,\dots,n$.
\end{proposition}
\begin{proof}
By \eqref{eq: upper bound} and \eqref{eq: boundedness lambda_n^eps} we can extract a subsequence from both the sequences $(\lambda_i^\eps)_{\eps>0}$ and $(\lambda_i(\Omega_\eps))_{\eps>0}$ such that
%\begin{align}
\[
\lambda_i^{\eps_k} \to \lambda_i,\ \ {\rm and}\ \
\lambda_i(\Omega_{\eps_k}) \to \widehat{\lambda}_i,
\]
%\end{align}
as $k\to \infty$, for all $i=1,\dots,n+1$.\\
By assumption we have $\lambda_i = \widehat{\lambda}_i$ for all $i=1,\dots,n$. Thus, by passing to the limit as $\eps \to 0$ in \eqref{eq: upper bound} (with $n$ replaced by $n+1$) and in \eqref{prop: lambda_n+1}, we get
\[ \lambda_n \leq \widehat{\lambda}_{n+1} - \delta \leq \lambda_{n+1} - \delta. \]

We rewrite $\lambda_1,\dots,\lambda_n$ without repetitions due to multiplicity in order to get a new sequence
\begin{equation} \label{proof: nonoverlappeigenvalues}
\widetilde{\lambda}_1< \widetilde{\lambda}_2<\dots< \widetilde{\lambda}_s = \lambda_n
\end{equation}
and set  $\widetilde{\lambda}_{s+1}:= \widehat{\lambda}_{n+1} \leq \lambda_{n+1}$. Thus, by assumption \eqref{prop: lambda_n+1} we have that
\begin{equation} \label{proof: nonoverlappeigenvalues2}
\widetilde{\lambda}_s < \widetilde{\lambda}_{s+1}.
\end{equation}
For each $r=1,\dots,s$, let $\widetilde{\lambda}_r = \lambda_{i_r} = \dots = \lambda_{j_r}$, for some $i_r \leq j_r$, $i_r, j_r \in \{1,\dots,n \}$, where it is understood that $j_r - i_r + 1$ is the multiplicity of $\widetilde{\lambda}_r$. Furthermore, we define the eigenprojector $Q_r$ from $L^2(\Omega_\eps)$ onto the linear span $[\varphi_{i_r}^\eps, \dots, \varphi_{j_r}^\eps]$ by
\begin{equation} \label{proof: def Q_r}
Q_r g = \sum_{i=i_r}^{j_r} (g, \varphi_{i_r}^\eps)_{L^2(\Omega_\eps)} \varphi_{i_r}^\eps.
\end{equation}
We now proceed to prove the following\\ \smallskip

\noindent\emph{Claim:} $\norma{\xi_i^{\eps_k} - Q_r \xi_i^{\eps_k}}_{H^2(\Omega_{\eps_k})} \to 0$ as $\eps \to 0$, for all $i_r \leq i \leq j_r$ and $r \leq s$.

\noindent Let us prove it by induction on $1 \leq r \leq s$.\\
If $r=1$, we define the function
\[
\chi_{\eps_k} = \xi_i^{\eps_k} - Q_1 \xi_i^{\eps_k} = \xi_i^{\eps_k} - \sum_{l=1}^{j_1} (\xi_i^{\eps_k}, \varphi_l^{\eps_k})_{L^2(\Omega_{\eps_k})} \varphi_l^{\eps_k}.
\]
Then $\chi_{\eps_k} \in H^2(\Omega_{\eps_k})$, $(\chi_{\eps_k} , \varphi_l^{\eps_k})_{L^2(\Omega_{\eps_k})}= 0$ for all $l=1,\dots, j_1$ and by the min-max representation of $\lambda_2(\Omega_{\eps_k})$ we have that
\begin{equation} \label{proof: bigger lambda_2}
[\chi_{\eps_k}]^2_{H^2_{\sigma,\tau}(\Omega_{\eps_k})} \geq \lambda_2(\Omega_{\eps_k)} \norma{\chi_{\eps_k}}^2_{L^2(\Omega_{\eps_k})} \geq \widetilde{\lambda}_2 \norma{\chi_{\eps_k}}^2_{L^2(\Omega_{\eps_k})} - o(1).
\end{equation}
On the other hand, it is easy to prove by definition of $\chi_{\eps_k}$ that
\begin{multline}
\int_{\Omega_{\eps_k}} (1-\sigma) \big(D^2\chi_{\eps_k} : D^2 \psi\big) + \sigma \Delta \chi_{\eps_k} \Delta\psi + \tau \nabla\chi_{\eps_k}\cdot \nabla \psi + \chi_{\eps_k}\psi \, dx\\
= \lambda_1(\Omega_{\eps_k})\int_{\Omega_{\eps_k}} \chi_{\eps_k} \psi \, dx + o(1)
\end{multline}
for all $\psi\in H^2(\Omega_{\eps_k})$. This in particular implies that
\begin{equation} \label{proof: chi_eps equality}
[\chi_{\eps_k}]^2_{H^2_{\sigma,\tau}(\Omega_{\eps_k})} = \lambda_1(\Omega_{\eps_k})  \norma{\chi_{\eps_k}}^2_{L^2(\Omega_{\eps_k})} +o(1)
\end{equation}
and consequently,
\begin{equation} \label{proof: less lambda_1}
[\chi_{\eps_k}]^2_{H^2_{\sigma,\tau}(\Omega_{\eps_k})} \leq \widetilde{\lambda}_1 \norma{\chi_{\eps_k}}^2_{L^2(\Omega_{\eps_k})} + o(1).
\end{equation}
Hence, inequalities \eqref{proof: bigger lambda_2}, \eqref{proof: less lambda_1} imply that
\[
\widetilde{\lambda}_2 \norma{\chi_{\eps_k}}^2_{L^2(\Omega_{\eps_k})} - o(1) \leq \widetilde{\lambda}_1 \norma{\chi_{\eps_k}}^2_{L^2(\Omega_{\eps_k})} + o(1),
\]
which implies that $\norma{\chi_{\eps_k}}_{L^2(\Omega_{\eps_k})} = o(1)$ (otherwise we would have $\widetilde{\lambda}_2 \leq \widetilde{\lambda}_1 + o(1)$, against \eqref{proof: nonoverlappeigenvalues}). Finally, equation \eqref{proof: chi_eps equality} implies that $[\chi_{\eps_k}]_{H^2_{\sigma,\tau}(\Omega_{\eps_k})} = o(1)$, so that also $\norma{\chi_{\eps_k}}_{H^2(\Omega_{\eps_k})}= o(1)$.

Let $r>1$ and assume by induction hypothesis that
\begin{equation} \label{proof: ind_hyp}
\norma{\xi^{\eps_k}_i - Q_t \xi^{\eps_k}_i}_{H^2(\Omega_{\eps_k})} \to 0
\end{equation}
as $k \to \infty$, for all $i_t \leq i \leq j_t$ and for all $t=1,\dots,r-1$. We have to prove that \eqref{proof: ind_hyp} holds also for $t=r$. Let $i_r \leq i \leq j_r$ and let $\chi_{\eps_k} = \xi_i^{\eps_k} - Q_r \xi_i^{\eps_k}$. Then
\begin{equation} \label{proof: almost orthog}
(\chi_{\eps_k}, \varphi_h^{\eps_k})_{L^2(\Omega_{\eps_k})} \to 0 \quad \text{as $k \to \infty$, for all $h=1,\dots,j_r$ }.
\end{equation}
Indeed, if $h \in \{i_r, \dots, j_r\}$ then by definition of $\chi_{\eps_k}$, $(\chi_{\eps_k}, \varphi_h^{\eps_k})_{L^2(\Omega_{\eps_k})} = 0$. Otherwise, if $h < i_r$, note that the function $\varphi_h^{\eps_k}$ satisfies
\begin{multline*}
\int_{\Omega_{\eps_k}} (1-\sigma) \left( D^2\varphi_h^{\eps_k} : D^2\psi \right) + \sigma \Delta \varphi_h^{\eps_k} \Delta \psi + \tau \nabla \varphi_h^{\eps_k} \nabla\psi + \varphi_h^{\eps_k} \psi\, dx \\
= \lambda_h(\Omega_{\eps_k}) \int_{\Omega_{\eps_k}} \varphi_h^{\eps_k} \psi\,dx\, ,
\end{multline*}
for all $\psi \in H^2(\Omega_{\eps_k})$, briefly
$
B_{\Omega_{\eps_k}}(\varphi_h^{\eps_k}, \psi) = \lambda_h(\Omega_{\eps_k})(\varphi_h^{\eps_k}, \psi)_{L^2(\Omega_{\eps_k})}\, ,
$
for all $\psi \in H^2(\Omega_{\eps_k})$, where $B_U$ denotes the quadratic form associated with the operator
$\Delta^2-\tau\Delta +I$
on an open set $U$.  Similarly,
$
B_{\Omega_{\eps_k}}(\xi_i^{\eps_k}, \psi) = \lambda_i^{\eps_k}(\xi_i^{\eps_k}, \psi)_{L^2(\Omega_{\eps_k})} + o(1)
$
for all $\psi \in H^2(\Omega_{\eps_k})$. Thus,
$\lambda_h(\Omega_{\eps_k})(\varphi_h^{\eps_k}, \xi_i^{\eps_k})_{L^2(\Omega_{\eps_k})} = \lambda_i^{\eps_k}(\xi_i^{\eps_k}, \varphi_h^{\eps_k})_{L^2(\Omega_{\eps_k})} + o(1)
$
which implies
\begin{equation}
\label{proof: difference eigenvalues}
( \lambda_h(\Omega_{\eps_k}) - \lambda_i^{\eps_k}) (\varphi_h^{\eps_k}, \xi_i^{\eps_k})_{L^2(\Omega_{\eps_k})} = o(1)
\end{equation}
and since $( \lambda_h(\Omega_{\eps_k}) - \lambda_i^{\eps_k}) \to (\widetilde{\lambda}_h - \widetilde{\lambda}_i) \neq 0$ by assumption, by \eqref{proof: difference eigenvalues} we deduce that $(\varphi_h^{\eps_k}, \xi_i^{\eps_k})_{L^2(\Omega_{\eps_k})} = o(1)$ as $\eps_k \to 0$, for all $h=1,\dots,j_r$, which implies \eqref{proof: almost orthog}.
%(\eqref{proof: ind_hyp} and the definition of $Q_r$).

As in the case $r=1$ we may deduce that
\begin{equation} \label{proof: bigger lambda_r+1}
[\chi_{\eps_k}]^2_{H^2_{\sigma,\tau}(\Omega_{\eps_k})} \geq \widetilde{\lambda}_{r+1} \norma{\chi_{\eps_k}}^2_{L^2(\Omega_{\eps_k})} - o(1).
\end{equation}
On the other hand, by definition of $\chi_{\eps_k}$ we have
\begin{equation} \label{proof: less lambda_r}
[\chi_{\eps_k}]^2_{H^2_{\sigma,\tau}(\Omega_{\eps_k})} \leq \widetilde{\lambda}_r \norma{\chi_{\eps_k}}^2_{L^2(\Omega_{\eps_k})} + o(1).
\end{equation}
By \eqref{proof: bigger lambda_r+1}, \eqref{proof: less lambda_r} and \eqref{proof: nonoverlappeigenvalues} it must be $\norma{\chi_{\eps_k}}^2_{L^2(\Omega_{\eps_k})} = o(1)$ and by \eqref{proof: less lambda_r} we deduce that $[\chi_{\eps_k}]^2_{H^2_{\sigma,\tau}(\Omega_{\eps_k})} = o(1)$, hence  $\norma{\chi_{\eps_k}}_{H^2(\Omega_{\eps_k})} \to 0$,
as $k\to \infty$. This concludes the proof of the Claim.\\

Now define the projector $\widetilde{Q}_n$ from $L^2(\Omega_\eps)$ into the linear span $[\varphi_1^{\eps}, \dots, \varphi_n^{\eps}]$ by
\[
\widetilde{Q}_n g = \sum_{i=1}^n (g,\varphi_i^\eps)_{L^2(\Omega_\eps)} \varphi_i^\eps.
\]
Then, as a consequence of the Claim we have that
\begin{equation}
\label{convergence xi}
\norma{\xi_i^{\eps_k} - \widetilde{Q}_n \xi_i^{\eps_k}}_{H^2(\Omega_{\eps_k})} \to 0
\end{equation}
as $k \to \infty$, for all $i=1,\dots,n$. Indeed for all indexes $i=1,\dots,n$ there exists $1 \leq r\leq s$ such that $i_r \leq i \leq j_r$; let assume for simplicity that $r=1$. Then we have $\norma{\xi_i^{\eps_k} - Q_1 \xi_i^{\eps_k}}_{H^2(\Omega_{\eps_k})} \to 0$ as $k\to \infty$; and also
\[
\norma{\xi_i^{\eps_k} - \widetilde{Q}_n \xi_i^{\eps_k}}_{H^2(\Omega_{\eps_k})} \leq \norma{\xi_i^{\eps_k} - Q_1 \xi_i^{\eps_k}}_{H^2(\Omega_{\eps_k})} + \sum_{l > j_1}^n \big\lvert(\xi_i^{\eps_k}, \varphi_l^{\eps_k})_{L^2(\Omega_{\eps_k})}\big\rvert \norma{\varphi_l^{\eps_k}}_{H^2(\Omega_{\eps_k})}
\]
and the right-hand side tends to 0 as $k \to \infty$ because $\norma{\varphi_l^{\eps_k}}_{H^2(\Omega_{\eps_k})}$ is uniformly bounded in $k$ and $(\xi_i^{\eps_k}, \varphi_l^{\eps_k})_{L^2(\Omega_{\eps_k})} \to 0$ as $k\to \infty$ (to see this it is sufficient to argue as in the proof of \eqref{proof: difference eigenvalues}). Moreover, since $\norma{\xi_i^{\eps_k} - \phi_i^{\eps_k}}_{H^2(\Omega) \oplus H^2(R_{\eps_k})} \to 0$ as $k \to \infty$ for all $i=1,\dots,n$, we also have $\norma{\phi_i^{\eps_k} - \widetilde{Q}_n \phi_i^{\eps_k}}_{H^2(\Omega) \oplus H^2(R_{\eps_k})} \to 0$
as $k \to \infty$, for all $i=1,\dots,n$. Thus $(\widetilde{Q}_n \phi_1^{\eps_k}, \dots,$ $ \widetilde{Q}_n \phi_n^{\eps_k} )$ is a basis in $(L^2(\Omega_{\eps_k})^n)$ for $[\varphi_1^{\eps_k}, \dots, \varphi_n^{\eps_k}]$. Hence,
$
\varphi_i^{\eps_k} = \sum_{l=1}^n a_{li}^{\eps_k} \widetilde{Q}_n \phi_l^{\eps_k}
$
for some coefficients $a_{li}^{\eps_k} = (\varphi_i^{\eps_k}, \phi_l^{\eps_k})_{L^2(\Omega_{\eps_k})} + o(1)$ as $k\to\infty$. Then for all $i =1,\dots,n$ we have
\begin{multline*}
\norma{\varphi_i^{\eps_k} - P_n \varphi_i^{\eps_k}}_{H^2(\Omega) \oplus H^2(R_{\eps_k}\!)}\\
 = \bigg\lVert \sum_{l=1}^n (\varphi_i^{\eps_k}, \phi_l^{\eps_k})_{L^2} [\phi_l^{\eps_k} - \widetilde{Q}_n \phi_l^{\eps_k}] + o(1) \sum_{l=1}^n \widetilde{Q}_n \phi_l^{\eps_k} \bigg \rVert_{H^2(\Omega) \oplus H^2(R_{\eps_k}\!)}
\end{multline*}
and the right-hand side tends to 0 as $k \to \infty$.
\end{proof}

\begin{remark}
\label{rmk: orthogonal matrix A}
In the previous proof one could prove that the matrix
$A = (a_{li} ^{\epsilon_k}    )_{l,i=1,\dots,n}  $
is almost orthogonal, in the sense that $A A^t = A^t A = \mathbb{I} + o(1)$ as $k \to \infty$. To prove this it is sufficient to show that the matrix
%\begin{equation}
%\label{A is orthogonal}
 $\tilde A= \bigl((\phi^{\eps_k}_l, \varphi^{\eps_k}_m)_{L^2(\Omega_{\eps_k})}\bigr)_{l,m=1,\dots,n}$
%\end{equation}
is almost orthogonal.   Let $l$ be fixed and note that
$
\phi^{\eps_k}_l = \sum_{m=1}^n (\phi^{\eps_k}_l, \varphi_m^{\eps_k})_{L^2(\Omega_{\eps_k})} \varphi^{\eps_k}_m + (\mathbb{I}-\widetilde{Q}_m) \phi^{\eps_k}_l,
$
hence, by \eqref{convergence xi} we deduce that
\begin{equation}\label{almost orthogonal matrix}
\delta_{li} = (\phi^{\eps_k}_l, \phi^{\eps_k}_i)_{L^2(\Omega_{\eps_k})} = \sum_{m=1}^n (\phi^{\eps_k}_l, \varphi^{\eps_k}_m)_{L^2(\Omega_{\eps_k})} (\varphi^{\eps_k}_m, \phi^{\eps_k}_i)_{L^2(\Omega_{\eps_k})} + o(1)\, ,
\end{equation}
as $k \to \infty$.
Note that we can rewrite \eqref{almost orthogonal matrix} as  $\tilde A \tilde A^t = \mathbb{I} + o(1)$, and in a similar way we also get  that $\tilde A^t \tilde A = \mathbb{I} + o(1)$, concluding the proof.
\end{remark}

In the sequel we shall need the following lemma.

\begin{lemma} \label{lemma: equation for chi} Let  $1 \leq i \leq j \leq n$.     Assume that $\widehat \lambda \in \R$ is such that, possibly passing to a subsequence,  $\lambda_m(\Omega_\epsilon )\to  \widehat{\lambda}$ as $\epsilon \to 0$ for all $m \in \{i, \dots, j \}$.
If  $\chi_{\eps} \in [\varphi_i^{\eps}, \dots, \varphi_j^{\eps}]$, $\norma{\chi_{\eps}}_{L^2(\Omega_{\eps})}=1$ and $\chi_{\eps}|_{\Omega} \rightharpoonup \chi$ in $H^2(\Omega)$
then
\begin{equation}
\label{eq: equation for chi}
\int_{\Omega} (1-\sigma) (D^2 \chi : D^2 \psi) + \sigma \Delta \chi \Delta \psi + \tau \nabla \chi \cdot\nabla \psi + \chi \psi\, dx = \widehat{\lambda} \int_{\Omega} \chi \psi\, dx\, ,
\end{equation}
for all $\psi \in H^2(\Omega)$.
\end{lemma}
\begin{proof}
Since $\chi_{\eps} \in [\varphi_i^{\eps}, \dots, \varphi_j^{\eps}]$ and $\norma{\chi_{\eps}}_{L^2(\Omega_{\eps})}=1$ there exist coefficients $(a_l(\eps))_{l=i}^j$ such that
$
\chi_{\eps} = \sum_{l=i}^j a_l(\eps) \varphi_l^{\eps}$ and $\sum_{l=i}^j a_l^2(\eps) = 1.$
Note that for all $m \in \{i, \dots, j \}$, possibly passing to a subsequence, there exists $\widehat{\varphi}_m \in H^2(\Omega)$ such that $\varphi_m^{\eps}|_{\Omega} \rightharpoonup \widehat{\varphi}_m$ in $H^2(\Omega)$. Since $\chi_{\eps}|_{\Omega} \rightharpoonup \chi$ in $H^2(\Omega)$ by assumption, we get that $\chi = \sum_{l=i}^j a_l \widehat{\varphi}_l$ in $\Omega$ for some coefficients $(a_l)_{l=i}^j$. Let $\psi \in H^2(\Omega)$ be fixed and consider an extension $\widetilde{\psi} = E \psi \in H^2(\R^N)$. Then
\begin{equation} \label{proof: lemma chi}
\begin{split}
&\int_{\Omega_{\eps}} (1-\sigma) \bigl(D^2\chi_{\eps} : D^2 \widetilde{\psi}\bigr) + \sigma \Delta \chi_{\eps} \Delta \widetilde{\psi} + \tau \nabla \chi_{\eps} \nabla \widetilde{\psi} + \chi_{\eps}\widetilde{\psi}\\
&= \sum_{l=i}^j a_l(\eps) \biggl[ \int_{\Omega_{\eps}} (1-\sigma) \bigl(D^2 \varphi_l^{\eps} : D^2 \widetilde{\psi}\bigr) + \sigma \Delta \varphi_l^{\eps} \Delta \widetilde{\psi} + \tau \nabla \varphi_l^{\eps} \nabla \widetilde{\psi} + \varphi_l^{\eps}\widetilde{\psi}  \biggr]\\
&= \sum_{l=i}^j a_l(\eps) \lambda_l(\Omega_{\eps}) \int_{\Omega_{\eps}} \varphi_l^{\eps} \widetilde{\psi}.
\end{split}
\end{equation}
Then it is possible to pass to the limit in both sides of \eqref{proof: lemma chi} by splitting the integrals over $\Omega_{\eps}$ into an integral over $R_{\eps}$ (that tends to $0$ as $\eps \to 0$) and an integral over $\Omega$. Moreover, the integrals over $\Omega$ will converge to the corresponding integrals in \eqref{eq: equation for chi} as $\eps \to 0$, because of the weak convergence of $\chi_{\eps}$ in $H^2(\Omega)$ and the strong convergence of $E\psi$ to $\psi$ in $H^2(\Omega)$.
\end{proof}

We proceed to prove the lower bound for $\lambda_n(\Omega_\eps)$.
To do so,  we need to add an extra assumption on the shape of $\Omega_{\epsilon}$. Hence,
  we introduce the following condition in the spirit of what is known for the Neumann Laplacian (see e.g., \cite{ArrPhD}, \cite{Arr1}, \cite{AHH}).

\begin{definition}[H-Condition]
\label{def: H condition}
We say that the family of dumbbell domains   $\Omega_\eps$, $\eps>0$,  satisfies the H-Condition if, given functions $u_\eps \in H^2(\Omega_\eps)$ such that $\norma{u_\eps}_{H^2(\Omega_\eps)} \leq R$ for all $\eps>0$, there exist functions  $\bar{u}_\eps \in H^2_{L_\eps}(R_\eps)$ such that
\begin{enumerate}[label=(\roman*)]
\item $\norma{u_\eps - \bar{u}_\eps }_{L^2(R_\eps)} \to 0$ as $\eps \to 0$,
\item $[\bar{u}_\eps]^2_{H^2_{\sigma, \tau}(R_\eps)} \leq [u_\eps]^2_{H^2_{\sigma, \tau}(\Omega_\eps)} + o(1)$ as $\eps \to 0$.
\end{enumerate}
\end{definition}

Recall that $[\cdot ]_{H^2_{\sigma,\tau}}$ is defined above in Definition \ref{definitionNorm}. We will show in Section \ref{sec: proof H condition regular dumbbells} that a wide class of channels $R_\eps$ satisfies the H-Condition.

\begin{theorem}[Lower bound] \label{thm: lower bound}
Assume that the family of dumbbell domains $\Omega_\eps$, $\eps>0$, satisfies the H-Condition.  Then for every $n\in \N$ we have $\lambda_n(\Omega_\eps) \geq \lambda_n^\eps - o(1)$ as $\eps \to 0$.
\end{theorem}
\begin{proof}
By Theorem \ref{thm: upper bound} and its proof we know that both  $\lambda_i(\Omega_\eps)$ and $\lambda_i^\eps$ are uniformly bounded in $\eps$. Then, for each subsequence $\eps_k$ we can find a subsequence (which we still call $\eps_k$), sequences of real numbers $(\lambda_i)_{i\in \N}$, $(\widehat{\lambda}_i)_{i\in \N}$, and sequences of $H^2(\Omega)$ functions $(\phi_i)_{i \in \N}$, $(\widehat{\varphi}_i)_{i \in \N}$, such that the following conditions are satisfied:
\begin{enumerate}[label=(\roman*)]
\item $\lambda_i^{\eps_k} \longrightarrow \lambda_i$, for all $i \geq 1$;
\item $\lambda_i(\Omega_{\eps_k}) \longrightarrow \widehat{\lambda}_i$, for all $i \geq 1$;
%\item $\phi_i^{\eps_k} |_{\Omega} \longrightarrow \phi_i $ strongly in $H^2(\Omega)$, for all $i \geq 1$ ;
\item $\xi^{\eps_k}_i|_{\Omega} \longrightarrow \phi_i $ strongly in $H^2(\Omega)$, for all $i \geq 1$;
\item $\varphi_i^{\eps_k}|_{\Omega} \longrightarrow \widehat{\varphi}_i$ weakly in $H^2(\Omega)$, for all $i \geq 1$;
%\item $x_{\eps_k} \to x$, $N(x_{\eps_k}) \to \bar{N}$.
\end{enumerate}
Note that $(iii)$ immediately follows by recalling that $\xi^{\eps_k}_i|_{\Omega}$ either it is zero or it coincides with $\varphi_i^\Omega$. Then $(iv)$ is deduced by the estimate
$\norma{\varphi_i^{\eps_k}}_{H^2(\Omega_{\eps_k})} \leq c\, \lambda_i(\Omega_{\eps_k})$ and by the boundedness of the sequence $\lambda_i(\Omega_{\eps_k})$,  $k \in \N$.

We plan to prove that $\widehat{\lambda}_i = \lambda_i$ for all $i\geq 1$. We do it by induction.
For $i=1$ we clearly have $\lambda_1 = \lambda_1(\Omega) = 1 = \lambda(\Omega_{\eps_k})$ for all $k$; hence, passing to the limit as $k\to\infty$ in the right-hand side of the former equality we get $\lambda_1 = \widehat{\lambda_1}$.
Then, we assume by induction hypothesis that $\widehat{\lambda}_i = \lambda_i$ for all $i=1,\dots,n$ and we prove that $\widehat{\lambda}_{n+1} = \lambda_{n+1}$. There are two possibilities: either $\lambda_n = \lambda_{n+1}$ or $\lambda_n < \lambda_{n+1}$. In the first case we deduce by \eqref{eq: upper bound} that
\[
\lambda_n = \widehat{\lambda}_n \leq \widehat{\lambda}_{n+1} \leq \lambda_{n+1} = \lambda_n,
\]
hence all the inequalities are equalities and in particular $\widehat{\lambda}_{n+1} = \lambda_{n+1}$.
Consequently we can assume without loss of generality that $\lambda_n < \lambda_{n+1}$. In this case we must have $\widehat{\lambda}_{n+1} \in [\lambda_n, \lambda_{n+1}]$ because $\lambda_n=\widehat{\lambda_n}$ and $\lambda_n(\Omega_{\eps_k}) \leq \lambda_{n+1}(\Omega_{\eps_k}) \leq \lambda_{n+1}^{\eps_k} + o(1)$ as $k\to \infty$. Let $r= \max\{\lambda_i : i<n, \lambda_i < \lambda_n \}$. Then $\lambda_r < \lambda_{r+1}= \dots = \lambda_n < \lambda_{n+1}$. In particular we can apply Proposition \ref{prop: convergence eigenprojections} with $n$ replaced by $r$ in order to get
\begin{equation} \label{proof: E5}
\norma{\varphi_i^{\eps_k} - P_r \varphi_i^{\eps_k}}_{H^2(\Omega) \oplus H^2(R_{\eps_k})} \to 0
\end{equation}
as $k \to \infty$, for all $i=1,\dots,r$. We now divide the proof in two steps.
\smallskip

\noindent\emph{Step 1}: we prove that $\lambda_n < \widehat{\lambda}_{n+1}$.\\
Let us assume by contradiction that $\lambda_n = \widehat{\lambda}_{n+1}$; then $\widehat{\lambda}_{r+1}=\dots=\widehat{\lambda}_n=\widehat{\lambda}_{n+1}$. Define the subspace $S$ by $S = [\varphi^{\eps_k}_{r+1}, \dots, \varphi^{\eps_k}_{n+1}]$. Hence, $S$ is $(n-r+1)$-dimensional. We then choose $\chi_{\eps_k} \in S$ with the following properties:
\begin{enumerate}[label=(\Roman*)]
\item $\norma{\chi_{\eps_k}}_{L^2(\Omega_{\eps_k})} = 1$.
\item $\chi_{\eps_k} \perp \phi^{\eps_k}_{r+1}, \dots, \phi^{\eps_k}_{n}$ in $L^2(\Omega_{\eps_k})$.
\end{enumerate}
This choice is possible because $[\phi^{\eps_k}_{r+1}, \dots, \phi^{\eps_k}_{n}]$ is $(n-r)$-dimensional. Moreover, we have that
\begin{equation} \label{proof: E6}
(\,\chi_{\eps_k}, \phi^{\eps_k}_i)_{L^2(\Omega_{\eps_k})} \longrightarrow 0
\end{equation}
as $k\to \infty$, for all $i=1, \dots, r$. To see this, recall that $\chi_{\eps_k} \in S$, hence
\begin{equation}
\label{proof: orthogonality chi_eps}
(\chi_{\eps_k}, \varphi^{\eps_k}_j)_{L^2(\Omega_{\eps_k})} = 0, \quad \forall j \leq r.
\end{equation}
By \eqref{proof: E5} and \eqref{proof: orthogonality chi_eps}, we have
\[
(\chi_{\eps_k}, P_r \varphi^{\eps_k}_j)_{L^2(\Omega_{\eps_k})} \longrightarrow 0, \quad \forall j \leq r
\]
as $k\to \infty$. Thus,
\begin{equation}
\label{proof: matrix times vector}
\sum_{l=1}^r (\varphi_j^{\eps_k}, \phi_l^{\eps_k})_{L^2(\Omega_{\eps_k})} (\chi_{\eps_k}, \phi_l^{\eps_k})_{L^2(\Omega_{\eps_k})} \longrightarrow 0, \quad \forall j \leq r,
\end{equation}
as $k\to \infty$. We can rewrite \eqref{proof: matrix times vector} as $A^t b \to 0$ as $k \to \infty$, where $A$ is the matrix defined in Remark \ref{rmk: orthogonal matrix A} and $b \in \R^r$ is the vector defined by $b_l = ((\chi_{\eps_k}, \phi_l^{\eps_k})_{L^2(\Omega_{\eps_k})})_l$ for all $l \in \{1,\dots,r\}$. Hence, also $A A^t b \to 0$ as $k \to \infty$ and by Remark \ref{rmk: orthogonal matrix A} we deduce that $A A^t b =\bigl(\mathbb{I} + o(1)\bigr) b = b + o(1) \to 0$ as $k \to \infty$, since $b$ is bounded in $k$. This implies that each component of $b$, which is $(\chi_{\eps_k}, \phi_l^{\eps_k})_{L^2(\Omega_{\eps_k})}$ tends to zero as $k \to \infty$, which is \eqref{proof: E6}.

\noindent It is now clear that \eqref{proof: E6} and property (II) of $ \chi_{\eps_k}$ yield
\begin{equation} \label{proof: E7}
(\chi_{\eps_k}, \phi^{\eps_k}_i)_{L^2(\Omega_{\eps_k})} \longrightarrow 0, \quad \text{for all $i=1,\dots,n$}
\end{equation}
as $k\to \infty$. Since $\norma{\chi_{\eps_k}}_{H^2(\Omega)} \leq C \max_{r+1 \leq j \leq n+1} \norma{\varphi_j^{\eps_k}}_{H^2(\Omega)} < \infty$ there exists a function $\chi \in H^2(\Omega)$     such that possibly passing to a subsequence
\begin{equation} \label{proof: E8}
\chi_{\eps_k} |_{\Omega} \rightharpoonup \chi \quad \text{in $H^2(\Omega)$},
\end{equation}
as $k\to \infty$. By \eqref{proof: E7} and \eqref{proof: E8} we deduce that $(\chi, \phi_i)_{L^2(\Omega)} = 0$, for all $i=1,\dots,n$. By Lemma \ref{lemma: equation for chi} $\chi$ is a $n$-th eigenfunction of $(\Delta^2 - \tau \Delta + \mathbb{I})_{N(\sigma)}$ in $\Omega$ associated with $\widetilde{\lambda}_n$ which is orthogonal to $\phi_1, \dots, \phi_n$, among which there are all the possible $n$-th eigenfunctions. Since $\lambda_n < \lambda_{n+1}$, the only way to avoid a contradiction is that $\chi \equiv 0$ in $\Omega$, that is
\begin{equation} \label{proof: E9}
\norma{\chi_{\eps_k}}_{L^2(\Omega)} \to 0, \quad\quad \norma{\chi_{\eps_k}}_{L^2(R_{\eps_k})} \to 1
\end{equation}
as $k\to \infty$. We use now the H-Condition; let us choose a sequence of functions $\overline{\chi}_{\eps_k} \in H^2_{L_{\eps_k}}(R_{\eps_k})$ such that $\norma{\chi_{\eps_k} - \overline{\chi}_{\eps_k}}_{L^2(R_{\eps_k})} \to 0$ as $k\to \infty$ and
\begin{equation}
\label{proof: ineq chibar}
[\overline{\chi}_{\eps_k}]^2_{H^2_{\sigma,\tau}(R_{\eps_k})} \leq [\chi_{\eps_k}]^2_{H^2_{\sigma,\tau}(\Omega_{\eps_k})} + o(1)
\end{equation}
as $k\to \infty$. Then we can extend by zero $\overline{\chi}_{\eps_k}$ to get a function (that we still call $\overline{\chi}_{\eps_k}$) in $H^2(\Omega_{\eps_k})$. Hence,
\begin{equation} \label{proof: eq}
\begin{split}
(\overline{\chi}_{\eps_k}, &\phi^{\eps_k}_i)_{L^2(\Omega_{\eps_k})}= (\overline{\chi}_{\eps_k}, \phi^{\eps_k}_i)_{L^2(R_{\eps_k})}\\
&= (\overline{\chi}_{\eps_k} - \chi_{\eps_k} , \phi^{\eps_k}_i)_{L^2(R_{\eps_k})} + (\chi_{\eps_k}, \phi^{\eps_k}_i)_{L^2(\Omega_{\eps_k})} - (\chi_{\eps_k}, \phi^{\eps_k}_i)_{L^2(\Omega)}
\end{split}
\end{equation}
for all $i=1,\dots,n$. By \eqref{proof: E7}, \eqref{proof: E9}, and the definition of $\overline{\chi}_{\eps_k}$ the right hand side of \eqref{proof: eq} tends to $0$ as $k\to \infty$, for all $i=1,\dots,n$. Thus, $\overline{\chi}_{\eps_k}$ is  asymptotically orthogonal
 to $\phi^{\eps_k}_1, \dots, \phi^{\eps_k}_n$. In particular, by the variational characterization of the eigenvalues $\lambda_i^{\eps_k}$ we get that
\begin{equation} \label{proof: E10}
[\overline{\chi}_{\eps_k}]^2_{H^2_{\sigma,\tau}(R_{\eps_k})} \geq \lambda^{\eps_k}_{n+1} \norma{\overline{\chi}_{\eps_k}}_{L^2(R_{\eps_k})} - o(1) \geq \lambda_{n+1} \norma{\overline{\chi}_{\eps_k}}_{L^2(R_{\eps_k})} - o(1)
\end{equation}
On the other hand, by \eqref{proof: ineq chibar} we deduce that
\[
\begin{split}
[\overline{\chi}_{\eps_k}]^2_{H^2_{\sigma,\tau}(R_{\eps_k})}&\leq [\chi_{\eps_k}]^2_{H^2_{\sigma,\tau}(\Omega_{\eps_k})} + o(1)\\
&=\lambda_n \norma{\chi_{\eps_k}}^2_{L^2(\Omega_{\eps_k})} + o(1) = \lambda_n \norma{\overline{\chi}_{\eps_k}}^2_{L^2(R_{\eps_k})} + o(1).
\end{split}
\]
This is a contradiction to \eqref{proof: E10} because $\lambda_n < \lambda_{n+1}$. Step 1 is complete.
\smallskip

\noindent \emph{Step 2}: we prove that $\widehat{\lambda}_{n+1} = \lambda_{n+1}$.\\
Assume by contradiction that $\widehat{\lambda}_{n+1} < \lambda_{n+1}$. Let us note that as a consequence of Step 1 we can use Proposition \ref{prop: convergence eigenprojections} for the $n$-th eigenvalues in order to obtain
\begin{equation} \label{proof: E11}
\norma{\varphi_i^{\eps_k} - P_n \varphi_i^{\eps_k}}_{H^2(\Omega) \oplus H^2(R_{\eps_k})} \to 0,
\end{equation}
for all $i =1,\dots,n$. Then we can use the same argument we used in Step 1 for $\chi_{\eps_k}$ to show that
\begin{equation} \label{proof: E12}
\norma{\varphi_{n+1}^{\eps_k}}_{L^2(\Omega)} \longrightarrow 0,
\end{equation}
as $k\to \infty$. To see this, just note that $\varphi_{n+1}^{\eps_k}$ is orthogonal to $\varphi_1^{\eps_k},\dots, \varphi_n^{\eps_k}$, and by \eqref{proof: E11} we deduce that $(\varphi_{n+1}^{\eps_k}, \phi_i^{\eps_k})_{L^2(\Omega_{\eps_k})} \to 0$, as $k \to \infty$, for all $i=1,\dots,n$. Moreover,
\begin{equation}
\label{proof: limit varphi_n+1}
(\varphi_{n+1}^{\eps_k}, \phi_{n+1}^{\eps_k})_{L^2(\Omega_{\eps_k})} \longrightarrow 0,
\end{equation}
as $k \to \infty$.  Indeed, looking  at the weak formulation of problem \eqref{PDE: main problem_eigenvalues} and denoting  by $B_U$ denotes the quadratic problem associated with the operator
$\Delta^2-\tau\Delta +I$
on an open set $U$,  we deduce both
\[
B_{\Omega_{\eps_k}}(\varphi_{n+1}^{\eps_k}, \phi_{n+1}^{\eps_k}) = \lambda_{n+1}(\Omega_{\eps_k})(\varphi_{n+1}^{\eps_k}, \phi_{n+1}^{\eps_k})_{L^2(\Omega_{\eps_k})} +o(1)
\]
and
\[
B_{\Omega_{\eps_k}}( \phi_{n+1}^{\eps_k},\varphi_{n+1}^{\eps_k}) = \lambda^{\eps_k}_{n+1}(\phi_{n+1}^{\eps_k}, \varphi_{n+1}^{\eps_k})_{L^2(\Omega_{\eps_k})} + o(1),
\]
and subtracting the above equalities and passing to the limit as $k \to \infty$ we obtain
$(\widehat{\lambda}_{n+1} - \lambda_{n+1}) \lim_{k \to \infty}(\phi_{n+1}^{\eps_k}, \varphi_{n+1}^{\eps_k})_{L^2(\Omega_{\eps_k})} = 0$, which implies \eqref{proof: limit varphi_n+1}. Then
\begin{equation}
\label{proof: varphi_n+1 almost orth}
(\varphi_{n+1}^{\eps_k}, \phi_i^{\eps_k})_{L^2(\Omega_{\eps_k})} \longrightarrow 0,
\end{equation}
as $k \to \infty$, for all $i=1,\dots,n +1$. Passing to the limit in $k$ we have $(\widehat{\varphi}_{n+1}, \phi_i)_{L^2(\Omega)} = 0$ for all $i=1,\dots,n+1$. However, as in Step 1 we would have $[\widehat{\varphi}_{n+1}]^2_{H^2_{\sigma,\tau}(\Omega)} = \widehat{\lambda}_{n+1} \norma{\widehat{\varphi}_{n+1}}_{L^2(\Omega)}$, which contradicts the assumption $\widehat{\lambda}_{n+1} < \lambda_{n+1}$ unless $\widehat{\varphi}_{n+1} \equiv 0$, which gives \eqref{proof: E12}.\\
Now we use the H-Condition and \eqref{proof: E12} in order to find a function $\overline{\varphi}^{\eps_k}_{n+1} \in H^2_{L_{\eps_k}}(R_{\eps_k})$ such that
$ \norma{\overline{\varphi}^{\eps_k}_{n+1}}_{L^2(R_{\eps_k})} = 1 + o(1)$ and
\[
[\overline{\varphi}^{\eps_k}_{n+1}]^2_{H^2_{\sigma,\tau}(R_{\eps_k})} \leq [\varphi^{\eps_k}_{n+1}]^2_{H^2_{\sigma,\tau}(\Omega_{\eps_k})} + o(1) = \lambda_{n+1}(\Omega_{\eps_k}) + o(1) \leq \widehat{\lambda}_{n+1} + o(1),
\]
as $k\to\infty$. On the other hand, by the variational characterization of $\lambda_{n+1}^{\eps_k}$ and by \eqref{proof: E12}, \eqref{proof: varphi_n+1 almost orth} we deduce that
$
[\overline{\varphi}^{\eps_k}_{n+1}]^2_{H^2_{\sigma,\tau}(R_{\eps_k})} \geq \lambda_{n+1}^{\eps_k}\norma{\overline{\varphi}^{\eps_k}_{n+1}}_{L^2(R_{\eps_k})} - o(1) \geq \lambda_{n+1} -o(1)
$
as $k \to \infty$, hence $\lambda_{n+1} \leq \widehat{\lambda}_{n+1}$, a contradiction. Thus it must be $\lambda_{n+1} = \widehat{\lambda}_{n+1}$.
\end{proof}

We will say that $x_\eps \in (0,\infty )$ \textit{divides the spectrum} of a family of  nonnegative self-adjoint operators $A_\eps$, $\epsilon >0$, with compact resolvents in $L^2(\Omega_\eps)$ if there exist $\delta, M, N, \eps_0 > 0$ such that
\begin{align}
[x_\eps - \delta, x_\eps + \delta] \cap \{ \lambda_n^\eps \}_{n=1}^\infty = \emptyset,& \quad\forall \eps < \eps_0\\
x_\eps \leq M,& \quad\forall \eps < \eps_0\\
N(x_{\eps}) := \#\{ \lambda_i^{\eps} : \lambda_i^{\eps} \leq x_\eps\}\leq N < \infty.
\end{align}
If $x_\eps$ divides the spectrum we  define the projector $P_{x_\eps}$ from $L^2(\Omega_\eps)$ onto the linear span $[\phi_1^{\eps}, \dots, \phi_{N(x_\eps)}^{\eps}]$ of the first $N(x_\eps)$ eigenfunctions  by
\[
P_{x_\eps} g = \sum_{i=1}^{N(x_\eps)} (g,\phi_i^\eps)_{L^2(\Omega_\eps)} \phi_i^\eps\, ,
\]
for all $g\in L^2(\Omega_\eps)$. Then, recalling Theorem \ref{thm: upper bound} and Theorem \ref{thm: lower bound} we deduce the following.

\begin{theorem}[(Decomposition of the eigenvalues)] \label{thm: eigenvalues decomposition}
Let  $\Omega_\eps$, $\eps>0$, be a family of  dumbbell domains satisfying the H-Condition. Then the following statements hold:
\begin{enumerate}[label =(\roman*)]
\item   $\lim_{\eps \to 0}\, \abs{\lambda_n(\Omega_\eps) - \lambda_n^\eps} = 0$,  for all  $n\in \N $.

\item   For any $x_\eps$ dividing the spectrum,
 $\lim_{\eps \to 0}\, \norma{\varphi^\eps_{r_\eps} - P_{x_\eps} \varphi^\eps_{r_\eps}}_{H^2(\Omega) \oplus H^2(R_\eps)} = 0$, for all $r_\eps = 1,\dots, N(x_\eps)$.\end{enumerate}
\end{theorem}

%-----------------------------END OF "DECOMPOSITION OF EIGENVALUES"--------------------------------------------------------------

\section{Proof of the H-Condition for regular dumbbells}
\label{sec: proof H condition regular dumbbells}
The goal of this section is to prove that the H-Condition holds for regular dumbbell domains. More precisely, we will consider channels $R_\eps$ such that the profile function $g$ has the following monotonicity property:\vspace{8pt}\\
 (MP): \textit{ there exists $\delta \in ]0, 1/2[$ such that $g$ is decreasing on $[0,\delta)$ and increasing on $(1-\delta, 1]$. }  \vspace{8pt}\\
\noindent If (MP) is satisfied then the set $A_{\eps} = \{ (x,y) \in \R^2 : x \in (0,\delta) \cup (1-\delta, 1), 0<y< \eps g(x)      \}$ is contained in the union of the two rectangles $[0, \delta] \times [0, \eps g(0)]$ and $[1-\delta, 1] \times [0, \eps g(1)]$. This fact will be used in the proof of the following theorem in order to control the $H^2$ norm of the candidate function $\overline{u}_\eps$ appearing in the H-Condition.

\begin{theorem} \label{thm: (MP) implies (H)}
The validity of condition (MP) implies the validity of the H-Condition.
\end{theorem}

Before writing the proof of this theorem we need to introduce some notation. First, for the sake of clarity we will consider a ``one-sided'' dumbbell $\Omega_\eps = \Omega \cup \overline{R_\eps}$ where $\Omega$ is a smooth bounded domain in $\R^2$ such that the segment $\{0\} \times [-1, 1]$ is contained in the boundary of $\Omega$, $\Omega \cap \{x \geq 0\} = \emptyset$ and $R_\eps$ is defined as in \eqref{def: R_eps}. We will assume that $R_\eps$ satisfies the (MP) condition on $0<x<\delta$ only. Let $L_\eps$ be the segment $\{0\} \times (0, \eps g(0))$.

For any  $\gamma \in (0,1)$, we define a function $\chi_\eps^\gamma \in C^{1,1}[-\eps^\gamma, 1]$,  such that   $\chi_\eps^\gamma(-\eps^\gamma) = (\chi_\eps^\gamma)'(-\eps^\gamma) = 0$,  $\chi_\eps^\gamma(x) \equiv 1$ for all $0 \leq x \leq 1$ and such that the following bounds on the derivatives
\[
\norma{(\chi_\eps^\gamma)'}_{L^\infty(-\eps^\gamma, 0)} \leq \frac{c_1}{\eps^\gamma}, \quad \norma{(\chi_\eps^\gamma)''}_{L^\infty(-\eps^\gamma, 0)} \leq \frac{c_2}{\eps^{2\gamma}},
\]
are satisfied for some positive real numbers $c_1, c_2$. A possible choice for $\chi_\eps^\gamma$ is
\[
\chi_\eps^\gamma(x) =
\begin{cases} -2 \bigg( \displaystyle \frac{x + \eps^\gamma}{\eps^\gamma} \bigg)^3 + 3 \bigg( \frac{x + \eps^\gamma}{\eps^\gamma} \bigg)^2, & x\in (-\eps^{\gamma},0), \\  \\
\qquad 1, & x\in (0,1),
\end{cases}
\]
which gives the (non-optimal) bounds $c_1 = 3/2$, $c_2 = 6$. For any  $\gamma, \beta>0$ we define the function $f_{\gamma, \beta} \in C^{1,1}(0,1)$ by setting
\begin{equation}
f=f_{\gamma,\beta}(x) =
\begin{cases} -\eps^\gamma \Big(\frac{x}{\eps^\beta}\Big)^2 + (\eps^\beta+2\eps^\gamma) \Big( \frac{x}{\eps^\beta} \Big) - \eps^\gamma, & x\in (0,\eps^\beta), \\
\qquad x, & x\in (\eps^\beta, 1).
\end{cases}
\end{equation}
Note that $f$ is a $C^{1,1}$-diffeomorphism from  $(0, \eps^\beta)$ onto $(-\eps^\gamma, \eps^\beta)$. Then,
\[
f'(x) =
\begin{cases} 1+2 \eps^{\gamma-\beta} \, (1-\frac{x}{\eps^\beta}), & x\in (0,\eps^\beta), \\
\qquad 1, & x\in (\eps^\beta, 1),
\end{cases}
\]
and
\[
f''(x) =
\begin{cases} - 2 \eps^{\gamma-2\beta}, & x\in (0,\eps^\beta), \\
\qquad 0, & x\in (\eps^\beta, 1),
\end{cases}
\]
which implies  that $|f'(x)-1|\leq 2 \eps^{\gamma-\beta}$, for all $x\in (0,1)$, and $|f''(x)|\leq 2  \eps^{\gamma-2\beta}$, for all $x\in (0,1)$. Thus,
 if  $\gamma>\beta$ then
\begin{equation}
\label{eq: asymptotics f'}
f'(x) = 1 + o(1) \quad \hbox{ as } \eps\to 0.
\end{equation}

For any  $\theta \in (0,1)$, we define the following sets:
\begin{align*}
&K_\eps^\theta = \{ (x,y) \in \Omega: - \eps^\theta < x < 0,\, 0 < y < \eps g(0) \}\, ,   \\
&\Gamma_\eps^\theta = \{ (-\eps^\theta, y) : 0<y<\eps g(0) \} \, ,         \\
&J_\eps^\theta = \{ (x,y) \in R_\eps : 0< x < \eps^\theta    \} \, ,   \\
&Q_\eps^\theta = \{(x,y) \in \R^2: 0<x<\eps^\theta, 0<y<\eps g(0)  \}\,  .
\end{align*}

Finally,  if  $\gamma /3 < \beta < \gamma /2$, for every  $u_\eps \in H^2(\Omega_\eps)$ we define  the function $\overline{u}_\eps \in H^2(R_\eps)$
by setting
\begin{equation} \label{def: u bar}
\overline{u}_\eps (x,y) = u_\eps(f(x),y) \,\chi^\gamma_\eps(f(x)),
\end{equation}
for all $(x,y) \in R_\eps$.   Function $\overline{u}_\eps$ will be used to prove the validity of
the   H-Condition. Before doing so, we need to prove the following proposition.

\bigskip

\begin{proposition}\label{prop: sym arg}
Let $\Omega_\eps = \Omega \cup R_\eps$ with $R_\eps$ satisfying the (MP) condition. Let $u_\eps \in H^2(\Omega_\eps)$ be such that $\norma{u_\eps}_ {H^2(\Omega_\eps)} \leq R$ for all $\eps > 0$. Then, with the notation above and for $0<\theta<\frac{1}{3}$, we have
\begin{equation}
 \norma{u_\eps}_{L^2(J_\eps^\theta)} = O(\eps^{2\theta}), \quad \norma{\nabla u_\eps}_{L^2(J_\eps^\theta)}=O(\eps^\theta), \hbox{  as  } \eps\to 0
\end{equation}
\end{proposition}
\begin{proof}
We define the function $u_\eps^s \in H^2(J_\eps^\theta)$ by setting
\[
u_\eps^s (x,y) = -3 u_\eps(-x,y) + 4 u_\eps \Bigl(-\frac{x}{2}, y \Bigr)
\]
for all $(x,y) \in J_\eps^\theta$. The function $u_\eps^s$ can be viewed as a higher order reflection of $u_\eps$ with respect to the $y$-axis. Let us note that we can estimate the $L^2$ norm of $u^s_\eps$, of its gradient and of its derivatives of order 2, in the following way:
\begin{align}
&\norma{u_\eps^s}_{L^2(J_\eps^\theta)} \leq C \norma{u_\eps}_{L^2(K_\eps^\theta)}, \label{proof: ineq 1}  \\
&\norma{\nabla u_\eps^s}_{L^2(J_\eps^\theta)} \leq C \norma{\nabla u_\eps}_{L^2(K_\eps^\theta)}, \label{proof: ineq 2} \\
&\norma{D^\alpha u_\eps^s}_{L^2(J_\eps^\theta)} \leq C\norma{D^\alpha u_\eps}_{L^2(K_\eps^\theta)}, \label{proof: ineq 3}
\end{align}
for any multiindex $\alpha$ of length $2$ and for some constant $C$ independent of $\eps$.  To obtain the three inequalities above, we are using that the image of $K_\eps^\theta$ under the reflexion about the $y$-axis contains $J_\eps^\theta$. This is a consequence of (MP).
Since the $L^2$ norms on the right-hand sides of the inequalities above are taken on a subset of $\Omega$, we can improve the estimate of  \eqref{proof: ineq 1}  and  \eqref{proof: ineq 2}  using H\"older's inequality and Sobolev embeddings to obtain
\begin{equation}\label{sobolev-1}
\norma{u_\eps}_{L^2(K_\eps^\theta)} \leq  |K_\eps^\theta|^{1/2} \norma{u_\eps}_{L^\infty(\Omega)} \leq c \bigl( \eps^{\theta + 1}\bigr)^{1/2} \norma{u_\eps}_{H^2(\Omega)}
\end{equation}
and in a similar way
\begin{equation}\label{sobolev-2}
\norma{\nabla u_\eps}_{L^2(K_\eps^\theta)} \leq  |K_\eps^\theta|^{\frac{1}{2} - \frac{1}{p}} \norma{\nabla u_\eps}_{L^p(\Omega)} \leq c \bigl(\eps^{\theta + 1}\bigr)^{\frac{1}{2} - \frac{1}{p}} \norma{u_\eps}_{H^2(\Omega)}
\end{equation}
for any $2<p<\infty$. Thus
\begin{equation}
\norma{u_\eps^s}_{L^2(J_\eps^\theta)} \leq C \eps^{\frac{\theta + 1}{2}} \norma{u_\eps}_{H^2(\Omega)},\ \ {\rm and}\ \
\norma{\nabla u_\eps^s}_{L^2(J_\eps^\theta)} \leq C \bigl(\eps^{\theta + 1}\bigr)^{\frac{1}{2} - \frac{1}{p}} \norma{u_\eps}_{H^2(\Omega)} \label{proof: decay ineq nabla u eps^s}.
\end{equation}

We also get
\begin{equation}\label{bound-second-derivatives}
\norma{D^\alpha u_\eps^s}_{L^2(J_\eps^\theta)} \leq C\norma{u_\eps}_{H^2(\Omega)}\, .
\end{equation}

\noindent We define now the function
\[
\psi_\eps = (u_\eps - u_\eps^s) |_{J_\eps^\theta}  \in H^2(J_\eps^\theta).
\]
Then $\psi_\eps = 0 = \nabla \psi_\eps$ on $L_\eps$. Let us first estimate $\norma{\nabla u_\eps}_{L^2(J_\eps^\theta)}$. Since we have
\[
\norma{\nabla u_\eps}^2_{L^2(J_\eps^\theta)} = \sum_{i=1}^2 \int_{J_\eps^\theta} \Big\lvert \frac{\partial u_\eps}{\partial x_i} \Big\rvert^2 dx,
\]
we can directly estimate the $L^2$-norm of the partial derivatives. %Let us consider a Lipschitz extension $E\psi_\eps \in H^2(Q_\eps)$ of the function $\psi_\eps$ to $Q_\eps$, which contains $J_\eps^\theta$ by hypothesis (MP).
Since $ \partial_{x_i} \psi_\eps = 0 $ on $L_\eps$ for all $i=1,2$ we apply a one-dimensional Poincaré inequality in the $x$-direction. We proceed as follows. For each $x_2\in (\eps g(\eps^\theta),\eps g(0))$ we denote by $h_\eps(x_2)$ the unique number such that $\eps g(h_\eps(x_2))=x_2$  (that is, the inverse function of $\eps g(\cdot)$,  which exists because of hypothesis (MP)). For $x_2\in (0,\eps g(\eps^\theta))$ we define $h_\eps(x_2)=\eps^{\theta}$.  Observe that $0\leq h_\eps(x_2)\leq \eps^\theta $  and that $J_\eps^\theta$ can be expressed as $J_\eps^\theta=\{(x_1,x_2):  0<x_2<\eps g(0); 0<x_1<h_\eps(x_2)\}$.  Hence, for $i=1,2$ we have
\begin{equation} \label{proof: Poincare ineq-1D}
\bigg \lVert \frac{\partial \psi_\eps}{\partial x_i}(\cdot,x_2) \bigg \rVert^2_{L^2(0,h_\eps(x_2))} \leq \frac{1}{\lambda_1(h_\eps(x_2))} \bigg \lVert   \frac{\partial}{\partial x_1} \bigg(\frac{\partial \psi_\eps}{\partial x_i}(\cdot, x_2) \bigg)\bigg \rVert^2_{L^2(0,h_\eps(x_2))}
\end{equation}
where $\lambda_1(\rho)=\bigl(\frac{\pi}{2}\bigr)^2 \rho^{-2}$ is the first eigenvalue of the problem
\[
\begin{cases}
&-v''=\lambda v, \quad \text{in $(0,\rho)$,}\\
&v(0)= 0,\\
&v'(\rho)= 0.
\end{cases}
\]
Since $0\leq h_\eps(x_2)\leq \eps^\theta$, we get the bound $\lambda_1(h_\eps(x_2))\geq \bigl(\frac{\pi}{2\eps^\theta }\bigr)^2$ and integrating in \eqref{proof: Poincare ineq-1D} with respect to  $x_2\in(0,\eps g(0))$, we get
\begin{equation} \label{proof: Poincare ineq}
\bigg \lVert \frac{\partial \psi_\eps}{\partial x_i} \bigg \rVert^2_{L^2(J_\eps^\theta)} \leq \biggl(\frac{2\eps^\theta }{\pi}\biggr)^2  \bigg \lVert \frac{\partial^2 \psi_\eps}{\partial x \partial x_i} \bigg \rVert^2_{L^2(J_\eps^\theta)}.
\end{equation}

\noindent Now note that
$
\Big\lVert \frac{\partial^2 \psi_\eps}{\partial x \partial x_i} \Big\rVert_{L^2(J_\eps^\theta)} \leq C \norma{u_\eps}_{H^2(\Omega_\eps)} \leq C R
$
for all $\eps > 0$, where we have used \eqref{bound-second-derivatives}. Hence we rewrite inequality \eqref{proof: Poincare ineq} in the following way:
\begin{equation} \label{proof: decay ineq psi_eps}
\Big \lVert \frac{\partial \psi_\eps}{\partial x_i} \Big \rVert_{L^2(J_\eps^\theta)} \leq \frac{2}{\pi} \eps^\theta  (C R + o(1)) = O(\eps^\theta)
\end{equation}
as $\eps \to 0$, for $i=1,2$.

Finally, by the inequalities  \eqref{proof: decay ineq nabla u eps^s}, \eqref{proof: decay ineq psi_eps} we deduce that
\begin{equation}
\begin{split}
\norma{\nabla u_\eps}_{L^2(J_\eps^\theta)} &\leq \norma{\nabla \psi_\eps}_{L^2(J_\eps^\theta)} + \norma{\nabla u^s_\eps}_{L^2(J_\eps^\theta)}\\
&\leq O(\eps^\theta) + C \bigl(\eps^{\theta + 1}\bigr)^{\frac{1}{2} - \frac{1}{p}} \norma{u_\eps}_{H^2(\Omega)}\leq O(\eps^\theta),
\end{split}
\end{equation}
where we  have used that $(\theta + 1) (1/2 - 1/p) > \theta$ for large enough $p$.

It remains to prove that $\norma{u_\eps}_{L^2(J_\eps^\theta)}= O(\eps^{2\theta})$ as $\eps \to 0$. We can repeat the argument for $u_\eps$ instead of $\partial_{x_i} u_\eps$, with the difference that now we can improve the decay of $\norma{\psi_\eps}_{L^2(J_\eps^\theta)}$ by using  the one-dimensional Poincar\'{e} inequality twice. More precisely we have that
\[
\norma{\psi_\eps}_{L^2(J_\eps^\theta)} \leq \Big(\frac{2}{\pi}\Big)^2 \eps^{2\theta} \bigg \lVert \frac{\partial^2 \psi_\eps}{\partial x^2} \bigg \rVert_{L^2(J_\eps^\theta)}
\]
from which we deduce
$
\norma{\psi_\eps}_{L^2(J_\eps^\theta)} = O(\eps^{2\theta})
$
as $\eps \to 0$. Hence,
\begin{equation}\label{eq:estimate ueps}
\begin{split}
\norma{u_\eps}_{L^2(J_\eps^\theta)} &\leq \norma{\psi_\eps}_{L^2(J_\eps^\theta)} + \norma{u^s_\eps}_{L^2(J_\eps^\theta)}
\leq O(\eps^{2\theta}) + C \eps^{\frac{\theta + 1}{2}} \norma{u_\eps}_{H^2(\Omega)} = O(\eps^{2\theta})
\end{split}
\end{equation}
as $\eps \to 0$, concluding the proof.
\end{proof}

We can now give a proof of Theorem \ref{thm: (MP) implies (H)}.
\begin{proof}[Proof of Theorem \ref{thm: (MP) implies (H)}]
Let $u_\eps\in H^2(\Omega_\eps)$ be such that $\norma{u_\eps}_{H^2(\Omega_\eps)}\leq R$ for any $\eps > 0$. We prove that the H-Condition holds if we choose $\overline{u}_\eps$ as in \eqref{def: u bar} with $\gamma <1/3$. Note that $u_\eps \equiv \overline{u}_\eps$ on $R_\eps \setminus J_\eps^\beta$. Let us first estimate $\norma{\overline{u}_\eps}_{L^2(J_\eps^\beta)}$. By a change of variable and by \eqref{eq: asymptotics f'} we deduce that
\begin{equation}
\begin{split}
\norma{\overline{u}_\eps}^2_{L^2(J_\eps^\beta)} &= \int_0^{\eps^\beta} \int_0^{\eps g(x)} |(u_\eps \chi_\eps^\gamma)(f(x),y)|^2\, dy dx\\
&= \int_{-\eps^\gamma}^{\eps^\beta} \int_0^{\eps g(f^{-1}(z))} |(u_\eps \chi_\eps^\gamma)(z,y)|^2 |f'(f^{-1}(z))|^{-1}\, dy dz\\
&\leq (1+o(1)) \int_{-\eps^\gamma}^{\eps^\beta} \int_0^{\eps g(f^{-1}(z))} |(u_\eps \chi_\eps^\gamma)(z,y)|^2 dy dz\\
&\leq (1+o(1)) \norma{u_\eps}^2_{L^2(Z_\eps^\gamma)},
\end{split}
\end{equation}
where $Z_\eps^\gamma = \{ (x,y) \in \Omega_\eps : -\eps^\gamma < x < \eps^\beta, 0 < y < \eps g(f^{-1}(x)) \}$. Note that since the function $g$ is non increasing, then $Z_\eps^\gamma\subset K_\eps^{\gamma}\cup J_\eps^\beta$. Hence,
\begin{equation} \label{proof: estimate overline(u)}
\norma{\overline{u}_\eps}_{L^2(J_\eps^\beta)}^2 \leq (1+o(1))( \norma{u_\eps}_{L^2(K^\gamma_\eps)}^2+ \norma{u_\eps}_{L^2(J_\eps^\beta)}^2).
\end{equation}
Note that the last summand in the right-hand side of \eqref{proof: estimate overline(u)} behaves as $O(\eps^{4\beta})$ as $\eps \to 0$ because of Proposition \ref{prop: sym arg}. Also by \eqref{sobolev-1} with $\theta$ replaced by $\gamma$, we get
\[
\norma{u_\eps}_{L^2(K^\gamma_\eps)} \leq c \eps^{\frac{\gamma+1}{2}} \norma{u_\eps}_{H^2(\Omega)},
\]

\noindent Thus,
\[
\norma{\overline{u}_\eps}_{L^2(J_\eps^\beta)}^2 \leq (1+o(1)) (O(\eps^{4\beta}) + O(\eps^{\gamma+1}) = O(\eps^{4\beta})
\]
as $\eps \to 0$. We then have by Proposition \ref{prop: sym arg} that
\[
\norma{u_\eps - \overline{u}_\eps}_{L^2(R_\eps)} = \norma{u_\eps - \overline{u}_\eps}_{L^2(J_\eps^\beta)} \leq \norma{u_\eps}_{L^2(J_\eps^\beta)} + \norma{\overline{u}_\eps}_{L^2(J_\eps^\beta)} = O(\eps^{2\beta})
\]
as $\eps \to 0$. This concludes the proof of $(i)$ in the H-Condition.

In order to prove $(ii)$  from Definition \ref{def: H condition}, we first need to compute $\norma{\nabla \overline{u}_\eps}_{L^2(J_\eps^\beta)}$ and $\norma{D^2 \overline{u}_\eps}_{L^2(J_\eps^\beta)}$. We have
\[
\begin{split}
&\frac{\partial  \overline{u}_\eps}{\partial x} (x,y)= \Bigg[\bigg(\frac{\partial u_\eps }{\partial x} \chi_\eps^\gamma\bigg) (f(x),y) + (u_\eps (\chi_\eps^\gamma)')(f(x),y) \Bigg] f'(x)\\
&\frac{\partial  \overline{u}_\eps}{\partial y} (x,y)= \bigg(\frac{\partial u_\eps}{\partial y} \chi_\eps^\gamma\bigg)(f(x),y).
\end{split}
\]
Hence,
\begin{equation}
\label{proof: gradient estimate}
\begin{split}
\norma{\nabla \overline{u}_\eps}_{L^2(J_\eps^\beta)} &\leq \norma{f'}_{L^\infty} \bigl( \norma{\nabla u_\eps (f(\cdot),\cdot)}_{L^2(J_\eps^\beta)} + \norma{(u_\eps (\chi_\eps^\gamma)')(f(\cdot),\cdot)}_{L^2(J_\eps^\beta)}\bigr)\\
&\leq  \norma{f'}_{L^\infty} \norma{f'}_{L^\infty}^{-1/2}\bigl(\norma{\nabla u_\eps}_{L^2(K_\eps^\gamma \cup J_\eps^\beta)} + c_1 \norma{\eps^{-\gamma} u_\eps}_{L^2(K_\eps^\gamma)}\bigr)\\
&\leq (1 + o(1)) \bigl(\norma{\nabla u_\eps}_{L^2(K_\eps^\gamma)} + \norma{\nabla u_\eps}_{L^2(J_\eps^\beta)} + c_1 \eps^{-\gamma} \norma{u_\eps}_{L^2(K_\eps^\gamma)}\bigr)
\end{split}
\end{equation}
where we have used the definition of $\chi_\eps^\gamma$ and the change of variables $(f(x), y) \mapsto (x, y)$. By Proposition \ref{prop: sym arg} we know that $\norma{\nabla u_\eps}_{L^2(J_\eps^\beta)} = O(\eps^{\beta})$ as $\eps \to 0$. Moreover, by \eqref{sobolev-1}, \eqref{sobolev-2} with $\theta$ replaced by $\gamma$, we deduce that
\[
\norma{u_\eps}_{L^2(K_\eps^\gamma)} = O(\eps^{  \frac{\gamma+1}{2}} ), \quad\quad \norma{\nabla u_\eps}_{L^2(K_\eps^\gamma)} = O(\eps^{\gamma_p}),
\]
for any $p < \infty$, where we have set
\[
\gamma_p = \biggl(\frac{1}{2} - \frac{1}{p}\biggr)(\gamma + 1).
\]
Finally, we deduce by \eqref{proof: gradient estimate} that
\begin{equation} \label{proof: gradient estimate final}
\norma{\nabla \overline{u}_\eps}_{L^2(J_\eps^\beta)} \leq (1+o(1)) (O(\eps^{\gamma_p}) + O(\eps^{\beta}) + \eps^{-\gamma} O(\eps^{\gamma_p}) ) = O(\eps^{\beta})
\end{equation}
because $\gamma_p-\gamma>\beta$,  for sufficiently large $p$ (note that  $\beta < (1-\gamma )/2$ for $\gamma < 1/3$).

We now estimate the $L^2$ norm of $D^2 \overline{u}_\eps$. In order to simplify our  notation we write $F(x,y) = (f(x),y)$, $\chi_\eps^\gamma = \chi$, $\bar u_\eps=\bar u$, $u_\eps=u$ and we use the subindex notation for the partial derivatives, that is, $u_x=\frac{\partial u}{\partial x}$ and so on. First, note that
\begin{equation}\label{2D-eq1}
\begin{split}
&\bar u_{xx} = \Big[\Big(u_{xx} \chi + 2 u_x \chi' + u \chi''\Big) \circ F \Big] \cdot |f'|^2 + \Big[\Big(u_x \chi + u \chi' \Big) \circ F\Big] \cdot f'', \\
&\bar u_{xy} = \Big[\Big( u_{xy} \chi + u_y \chi' \Big) \circ F \Big] \cdot f', \\
&\bar u_{yy} =\Big( u_{yy} \chi \Big) \circ F,
\end{split}
\end{equation}
and we may write
\begin{equation*}
\bar u_{xx} = [u_{xx} \chi \circ F]\cdot |f'|^2+ R_1,  \quad  \bar u_{xy} = [u_{xy} \chi \circ F] \cdot f'+ R_2, \quad \bar u_{yy} = u_{yy} \chi  \circ F.
\end{equation*}
where
\begin{equation*}
\begin{split}
&R_1= \Big[\Big( 2 u_x \chi' + u \chi''\Big) \circ F \Big] \cdot |f'|^2 + \Big[\Big(u_x \chi + u \chi' \Big) \circ F\Big] \cdot f'', \\
&R_2 =  u_y \chi' \circ F \cdot f'.
\end{split}
\end{equation*}

We now show that $\|R_1\|_{L^2(J_\eps^\beta)}=o(1)$, $\|R_2\|_{L^2(J_\eps^\beta)}=o(1)$ as $\eps\to 0$. For this, we will prove that each single term in $R_1$ and $R_2$ is $o(1)$ as $\eps\to 0$.  Recall that $f'(x)=1+o(1)$ and $f''(x)=o(1)$, $\chi'=O(\eps^{-\gamma})$ and $\chi''=O(\eps^{-2\gamma})$  for $x\in (0,\eps^\beta)$. By a change of variables, by the Sobolev Embedding Theorem and the definition of $\chi$ it is easy to deduce that
\begin{align*}
&\norma{(u_x \chi')\circ F}_{L^2(J_\eps^\beta)} \leq (1+o(1)) \norma{u_x \chi'}_{L^2(K^\gamma_\eps)} \leq C R \eps^{\gamma_p - \gamma} = O(\eps^\beta)\\
&\norma{(u \chi'')\circ F}_{L^2(J_\eps^\beta)} \leq c_2 (1+o(1)) \norma{u \eps^{-2\gamma}}_{L^2(K_\eps^\gamma)} \leq C R \eps^{\frac{1-3\gamma}{2}}\\
&\norma{(u_y \chi')\circ F}_{L^2(J_\eps^\beta)} \leq c_1 (1+o(1))  \norma{\eps^{-\gamma} u_y}_{L^2(K^\gamma_\eps)} \leq C R \eps^{\gamma_p - \gamma}= O(\eps^\beta)\, .
\end{align*}
By \eqref{proof: gradient estimate final} we also have
\begin{equation}
\norma{(u_x \chi + u \chi' ) \circ F}_{L^2(J_\eps^\beta)} \leq (1+o(1)) \norma{\nabla \overline{u}_\eps}_{L^2(J_\eps^\beta)} = O(\eps^\beta).
\end{equation}
Hence the $L^2$ norms of  $R_1$, $R_2$ vanish as $\eps \to 0$. In particular,
\begin{equation*}
\norma{D^2 \overline{u}_\eps}_{L^2(J_\eps^\beta)} =  (1 + o(1))\norma{D^2 u_\eps}_{L^2(K_\eps^\gamma \cup J_\eps^\beta)} + O(\eps^{\frac{1-3\gamma}{2}}) + O(\eps^\beta),
\end{equation*}
as $\eps\to 0$. In a similar way we can also prove that
\begin{equation*}
\norma{\Delta \overline{u}_\eps}_{L^2(J_\eps^\beta)}  =   (1 + o(1))\norma{\Delta u_\eps}_{L^2(K_\eps^\gamma \cup J_\eps^\beta)} + O(\eps^{\frac{1-3\gamma}{2}}) + O(\eps^\beta),
\end{equation*}
as $\eps\to 0$. Hence,
\begin{multline}
\label{proof: channel energy}
(1-\sigma) \norma{D^2 \overline{u}_\eps}^2_{L^2(J_\eps^\beta)} + \sigma \norma{\Delta \overline{u}_\eps}^2_{L^2(J_\eps^\beta)} + \tau \norma{\nabla \overline{u}_\eps }^2_{L^2(J_\eps^\beta)}\\
  =(1-\sigma)\norma{D^2 u_\eps}^2_{L^2(K_\eps^\gamma \cup J_\eps^\beta)} + \sigma \norma{\Delta u_\eps}^2_{L^2(K_\eps^\gamma \cup J_\eps^\beta)} + o(1).
\end{multline}
By adding to both handsides of \eqref{proof: channel energy} {\small$(1-\sigma)\norma{D^2 \overline{u}_\eps}^2_{L^2(R_\eps \setminus J_\eps^\beta)}$, $\sigma \norma{\Delta \overline{u}_\eps}^2_{L^2(R_\eps \setminus J_\eps^\beta)}$} and the lower order term $\tau \norma{\nabla \overline{u}_\eps }^2_{L^2(R_\eps \setminus J_\eps^\beta)}$, and keeping in account that $\overline{u}_\eps \equiv u_\eps$ on $R_\eps \setminus J_\eps^\beta$ we deduce that
\begin{multline}\label{mon}
(1-\sigma) \norma{D^2 \overline{u}_\eps}^2_{L^2(R_\eps)} + \sigma \norma{\Delta \overline{u}_\eps}^2_{L^2(R_\eps)} + \tau \norma{\nabla \overline{u}_\eps }^2_{L^2(R_\eps)}\\
  = (1-\sigma)\norma{D^2 u_\eps}^2_{L^2(K_\eps^\gamma \cup R_\eps)} + \sigma\norma{\Delta u_\eps}^2_{L^2(K_\eps^\gamma \cup R_\eps)} + \tau \norma{\nabla u_\eps }^2_{L^2(R_\eps \setminus J_\eps^\beta)} + o(1)\\
\leq  (1-\sigma)\norma{D^2 u_\eps}^2_{L^2(\Omega_\eps)} + \sigma \norma{\Delta u_\eps}^2_{L^2(\Omega_\eps)} + \tau \norma{\nabla u_\eps }^2_{L^2(\Omega_\eps)} + o(1),
\end{multline}
as $\eps \to 0$, concluding the proof of $(ii)$ in the H-Condition.
 Note that  in \eqref{mon}, we have used the monotonicity of the quadratic form with respect to inclusion of sets. Such property is straightforward
for $\sigma \in [0,1)$. In the case $\sigma \in (-1,0)$ it follows by observing that
\begin{multline*}
(1-\sigma) \bigl[ u^2_{xx} + 2 u^2_{xy} + u^2_{yy} \bigr] + \sigma \bigl[ u^2_{xx} + 2 u_{xx} u_{yy} + u^2_{yy} \bigr]\\
 \geq u^2_{xx} + u^2_{yy} + \sigma (u^2_{xx} + u^2_{yy} ) = (1+\sigma) (u^2_{xx} + u^2_{yy} ) > 0,
\end{multline*}
for all $u \in H^2(\Omega_\eps)$.
\end{proof}

%---------------------------------------END OF "Proof of the H-Condition for regular dumbbells" SECTION-----------------------------

\section{Asymptotic analysis on the thin domain}
\label{sec: thin plates}
The purpose of this section is to study the convergence of the eigenvalue  problem \eqref{PDE: R_eps} as $\eps \to 0$. Since the thin domain $R_\eps$ is shrinking to the segment $(0,1)$ as $\eps \to 0$, we plan to identify the  limiting problem in $(0,1)$ and to prove that the resolvent operator of problem \eqref{PDE: R_eps} converges as $\epsilon \to 0$  to the resolvent operator of the limiting problem in a suitable sense which guarantees the spectral convergence.

More precisely, we shall prove that the  the limiting eigenvalue problem in $[0,1]$ is
\begin{equation}\label{classiceigenode}
\begin{cases}
\frac{1-\sigma^2}{g} (gh'')''- \frac{\tau}{g}(gh')' + h = \theta h, &\text{in $(0,1)$,}\\
h(0)=h(1)=0,&\\
h'(0)=h'(1)=0.&
\end{cases}
\end{equation}
Note that the weak formulation of (\ref{classiceigenode}) is
\[
(1-\sigma^2)\int_0^1  h''\psi''gdx+\tau \int_0^1h'\psi'gdx+\int_0^1h\psi g dx=\theta \int_0^1h\psi g dx,
\]
for all $\psi\in H^2_0(0,1)$, where $h$ is to be found in the Sobolev space $H^2_0(0,1)$. In the sequel, we shall denote by $L^2_g(0,1)$ the Hilbert space $L^2((0,1); g(x)dx)$.

\subsection{Finding the limiting problem}
\label{subsection: finding limit prb}

In order to use  thin domain techniques in the spirit of \cite{HR}, we need to fix a reference domain $R_1$ and pull-back the eigenvalue problem defined on $R_{\eps}$ onto $R_1$ by means of a suitable diffeomorphism.

Let $R_1$ be the rescaled domain obtained by setting $\eps = 1$ in the definition of $R_\eps$ (see \eqref{def: R_eps}). For any fixed $\eps >0$, let  $\Phi_\eps$ be the map from $R_1$ to $R_\eps$ defined by $\Phi_\eps(x',y') = (x', \eps y')= (x,y)$ for all $(x',y') \in R_1$. We consider the composition operator $T_\eps$ from $L^2(R_\eps; \eps^{-1}dxdy)$ to $L^2(R_1)$ defined by
\[
T_\eps u(x',y') = u \circ \Phi_\eps (x', y') = u(x', \eps y')\, ,
\]
for all $u \in L^2(R_\eps)$, $(x',y') \in R_1$.  We also endow the spaces $H^2(R_1)$ and $H^2(R_\eps)$    with the norms  defined by
\begin{multline}
\|\varphi\|_{H^2_{\eps, \sigma, \tau}(R_1)}^2 =\int_{R_1} \Bigg((1-\sigma) \Bigg[ \abs*{\frac{\partial^2 \varphi}{\partial x^2}}^2 + \frac{2}{\eps^2}\abs*{\frac{\partial^2 \varphi}{\partial x \partial y}}^2 + \frac{1}{\eps^4} \abs*{\frac{\partial^2 \varphi}{\partial y^2}}^2 \Bigg]\\
+ \sigma \abs*{\frac{\partial^2 \varphi}{\partial x^2} + \frac{1}{\eps^2}\frac{\partial^2 \varphi}{\partial y^2}}^2  + \tau \Bigg[ \abs*{\frac{\partial \varphi}{\partial x}}^2 + \frac{1}{\eps}\abs*{\frac{\partial \varphi}{\partial y}}^2 \Bigg] + \abs{\varphi}^2\, \Bigg)dxdy\, ,
\end{multline}

\begin{multline}
\|\varphi\|_{H^2_{\sigma, \tau}(R_\eps)}^2 =\int_{R_\eps} \Bigg((1-\sigma) \Bigg[ \abs*{\frac{\partial^2 \varphi}{\partial x^2}}^2 + 2\abs*{\frac{\partial^2 \varphi}{\partial x \partial y}}^2 +  \abs*{\frac{\partial^2 \varphi}{\partial y^2}}^2 \Bigg]\\
+ \sigma \abs*{\frac{\partial^2 \varphi}{\partial x^2}+ \frac{\partial^2 \varphi}{\partial y^2}}^2 + \tau \Bigg[ \abs*{\frac{\partial \varphi}{\partial x}}^2 + \abs*{\frac{\partial \varphi}{\partial y}}^2 \Bigg] + \abs{\varphi}^2\,\Bigg) dxdy\, .
\end{multline}
It is not difficult to see that if $\varphi\in H^2(R_\eps)$ then
\[\|T_\eps \varphi\|_{H^2_{\eps,\sigma,\tau}(R_1)}^2= \eps^{-1} \|\varphi\|_{H^2_{\sigma,\tau}(R_\eps)}^2.\]

We consider  the following Poisson problem with datum $f_\eps \in L^2(R_\eps)$:
\begin{equation} \label{PDE: R_eps f_eps}
\begin{cases}
\Delta^2 v_\eps - \tau \Delta v_\eps + v_\eps = f_\eps, &\text{in $R_\eps$},\\
(1-\sigma) \frac{\partial^2 v_\eps}{\partial n_\eps^2} + \sigma \Delta v_\eps = 0, &\textup{on $\Gamma_\eps$},\\
\tau \frac{\partial v_\eps}{\partial n_\eps} - (1-\sigma) \Div_{\partial \Omega_\eps}(D^2v_\eps \cdot n_\eps)_{\partial \Omega_\eps} - \frac{\partial(\Delta v_\eps)}{\partial n_\eps} = 0, &\textup{on $\Gamma_\eps$,}\\
v = 0 = \frac{\partial v_\eps}{\partial n_\eps}, &\text{on $L_\eps$.}
\end{cases}
\end{equation}
Note that the energy space associated with Problem \eqref{PDE: R_eps f_eps} is exactly $H^2_{L_{\eps}}(R_\eps)$.
By setting  $\tilde{v}_\eps = v_\eps (x', \eps y')$, $\tilde{f}_\eps = f(x', \eps y')$ and  pulling-back problem   (\ref{PDE: R_eps f_eps}) to $R_1$ by means of  $\Phi_\eps$, we get the following equivalent problem in $R_1$ in the unknown  $\tilde{v}_\eps $ (we use again the variables $(x,y)$ instead of $(x',y')$ to simplify the notation):
{\small \begin{equation} \label{PDE: R_1}
\begin{cases}
\frac{\partial^4 \tilde{v}_\eps}{\partial x^4} + \frac{2}{\eps^2} \frac{\partial^4 \tilde{v}_\eps}{\partial x^2 \partial y^2} + \frac{1}{\eps^4} \frac{\partial^4 \tilde{v}_\eps}{\partial y^4} - \tau \Big( \frac{\partial^2 \tilde{v}_\eps}{\partial x^2} + \frac{1}{\eps^2} \frac{\partial^2 \tilde{v}_\eps}{\partial y^2} \Big) + \tilde{v}_\eps = \tilde{f}_\eps, &\text{in $R_1$},\\
(1-\sigma) \Big( \frac{\partial^2 \tilde{v}_\eps}{\partial x^2}\tilde{n}_x^2 + \frac{2}{\eps} \frac{\partial^2 \tilde{v}_\eps}{\partial x \partial y}\tilde{n}_x \tilde{n}_y + \frac{1}{\eps^2}\frac{\partial^2 \tilde{v}_\eps}{\partial y^2}\tilde{n}_y^2 \Big)  + \sigma \Big( \frac{\partial^2 \tilde{v}_\eps}{\partial x^2} + \frac{1}{\eps^2} \frac{\partial^2 \tilde{v}_\eps}{\partial y^2} \Big)= 0, &\textup{on $\Gamma_1$},\\
\tau \Big(\frac{\partial \tilde{v}_\eps}{\partial x} \tilde{n}_x + \frac{1}{\eps} \frac{\partial \tilde{v}_\eps}{\partial y} \tilde{n}_y \Big)  - (1-\sigma) \Div_{\Gamma_{1,\eps}}(D_\eps^2 \tilde{v}_\eps \cdot \tilde{n}){_{\Gamma_{1,\eps}} }- \nabla_\eps(\Delta_\eps \tilde{v}_\eps) \cdot \tilde{n} = 0, &\textup{on $\Gamma_{1}$,}\\
\tilde{v}_\eps = 0 = \frac{\partial \tilde{v}_\eps}{\partial x} n_x + \frac{1}{\eps}\frac{\partial \tilde{v}_\eps}{\partial y} \tilde{n}_y , &\text{on $L_1$.}
\end{cases}
\end{equation}}
Here   $\tilde{n} = (\tilde{n}_x, \tilde{n}_y) = (n_x, \eps^{-1}n_y)$ and  the operators $\Delta_\eps, \nabla_\eps$ are the standard differential operators associated with $(\partial_x, \eps^{-1} \partial_y)$. Moreover,
\[\Div_{\Gamma_{1,\eps}} F = \frac{\partial F_1}{\partial x} + \frac{1}{\eps}\frac{\partial F_2}{ \partial y} - \tilde{n}_\eps \nabla_{\!\eps} F\, \tilde{n}_\eps\]
and $(F)_{\Gamma_{1,\eps}} = F - (F, \tilde{n})\, \tilde{n}$ for any vector field $F =(F_1,F_2)$.

Assume now that   the data  $f_{\epsilon }$, $\eps>0$ are such that    $(\tilde{f}_\eps)_{\eps>0}$ is an equibounded family in $L^2(R_1)$, i.e.,
\begin{equation} \label{hypotesis on f eps}
\int_{R_1} \abs{\tilde{f}_\eps}^2\,dxdy' \leq c, \quad \textup{or equivalently} \quad \int_{R_\eps} \abs*{f_\eps}^2 dxdy \leq c \eps\, ,
\end{equation}
for all $\eps>0$, where  $c$  is a positive  constant not depending on $\eps$.

We plan to pass to the limit in  \eqref{PDE: R_1} as $\eps \to 0$ by arguing as follows. If $\tilde{v}_\eps \in H^2_{L_1}(R_1)$ is the solution to problem \eqref{PDE: R_1}, then we have the following integral equality
\small
\begin{multline}
\label{eq: weak formulation R1}
(1-\sigma) \int_{R_1} \frac{\partial^2 \tilde{v}_\eps }{\partial x^2} \frac{\partial^2 \varphi }{\partial x^2} + \frac{2}{\eps^2} \frac{\partial^2 \tilde{v}_\eps }{\partial x \partial y} \frac{\partial^2 \varphi }{\partial x \partial y} + \frac{1}{\eps^4} \frac{\partial^2 \tilde{v}_\eps }{\partial y^2} \frac{\partial^2 \varphi }{\partial y^2} dx\\
+ \sigma \int_{R_1} \Big(\frac{\partial^2 \tilde{v}_\eps }{\partial x^2} + \frac{1}{\eps^2} \frac{\partial^2 \tilde{v}_\eps }{\partial y^2}  \Big) \Big( \frac{\partial^2 \varphi }{\partial x^2} + \frac{1}{\eps^2} \frac{\partial^2 \varphi }{\partial y^2}     \Big) dx\\
 + \tau \int_{R_1}  \frac{\partial \tilde{v}_\eps }{\partial x} \frac{\partial \varphi }{\partial x} + \frac{1}{\eps^2} \frac{\partial \tilde{v}_\eps }{\partial y} \frac{\partial \varphi }{\partial y} dx + \int_{R_1} \tilde{v}_\eps \varphi dx = \int_{R_1} \tilde{f}_\eps \varphi dx
\end{multline}
\normalsize
for all $\varphi \in H^2_{L_1}(R_1)$. By choosing $\varphi = \tilde{v}_\eps$ we deduce the following apriori estimate:
\small
\begin{multline}
\label{ineq: apriori ineq tilde v eps}
(1-\sigma) \int_{R_1} \abs*{\frac{\partial^2 \tilde{v}_\eps }{\partial x^2}}^2 + \frac{2}{\eps^2} \abs*{\frac{\partial^2 \tilde{v}_\eps }{\partial x \partial y}}^2 + \frac{1}{\eps^4} \abs*{\frac{\partial^2 \tilde{v}_\eps }{\partial y^2}}^2 dx + \sigma \int_{R_1} \abs*{\frac{\partial^2 \tilde{v}_\eps }{\partial x^2} + \frac{1}{\eps^2} \frac{\partial^2 \tilde{v}_\eps }{\partial y^2}}^2  dx\\
+ \tau \int_{R_1} \abs*{\frac{\partial \tilde{v}_\eps }{\partial x}}^2 + \frac{1}{\eps^2} \abs*{\frac{\partial \tilde{v}_\eps }{\partial y}}^2 dx + \int_{R_1} \abs{\tilde{v}_\eps}^2 dx \leq \frac{1}{2} \int_{R_1} \abs{\tilde{f}_\eps}^2\,dx + \frac{1}{2} \int_{R_1} \abs{\tilde{v}_\eps}^2\,dx
\end{multline}
\normalsize
for all $\eps>0$. This implies that $\norma{\tilde{v}_\eps}_{H^2_{\eps,\sigma, \tau}(R_1)} \leq C$ for all $\eps> 0$, in particular $\norma{\tilde{v}_\eps}_{H^2(R_1)} \leq C(\sigma, \tau)$ for all $\eps>0$; hence, there exists $v \in H^2(R_1)$ such that, up to a subsequence
$\tilde{v}_\eps \to v$, weakly in $H^2(R_1)$, strongly in $H^1(R_1)$. Moreover from \eqref{ineq: apriori ineq tilde v eps} we deduce that
\begin{align}
\label{ineq: decay y derivatives 1} &\norma*{\frac{\partial^2 \tilde{v}_\eps}{\partial x \partial y}}_{L^2(R_1)} \leq C \eps,    \quad\quad  \norma*{\frac{\partial \tilde{v}_\eps}{\partial y}}_{L^2(R_1)} \leq C \eps , \\
\label{ineq: decay y derivatives 2}&\norma*{\frac{\partial^2 \tilde{v}_\eps}{\partial y^2}}_{L^2(R_1)} \leq C \eps^2,   \end{align}
for all $\eps > 0$, hence there exists $u \in L^2(R_1)$ such that, up to a subsequence
\begin{equation}\label{weakly-to-u}
\frac{1}{\eps^2} \frac{\partial^2 \tilde{v}_\eps}{\partial y^2} \rightharpoonup u, \hbox{ weakly in }L^2(R_1)
\end{equation}
 as $\eps \to 0$. By \eqref{ineq: decay y derivatives 1} we deduce that the limit function $v$ is constant in $y$. Indeed, if we choose any function $\phi \in C^\infty_c(R_1)$, then
\[
\int_{R_1} v \frac{\partial \phi}{ \partial y} = \lim_{\eps \to 0} \int_{R_1} \tilde{v}_\eps \frac{\partial \phi}{ \partial y} = - \lim_{\eps \to 0} \int_{R_1} \frac{\partial \tilde{v}_\eps}{\partial y} \phi = 0,
\]
hence $\frac{\partial v}{\partial y} = 0$ and then $v(x,y) \equiv v(x)$ for almost all $(x,y) \in R_1$. This suggests to choose test functions $\psi$ depending only on $x$ in the weak formulation \eqref{eq: weak formulation R1}. Possibly passing to a subsequence, there exists $f \in L^2(R_1)$ such that
\[
\tilde{f}_\eps \rightharpoonup f \quad \quad \text{in $L^2(R_1)$, as $\eps \to 0$}.
\]
Let $\psi \in H^2_0(0,1)$. Then $\psi \in H^2(R_1)$ (here it is understood that the function is extended to the whole of $R_1$ by setting $\psi(x,y) = \psi(x)$ for all $(x,y) \in R_1$) and clearly $\psi \equiv 0$ on $L_1$. Use $\psi$ as a test function in \eqref{eq: weak formulation R1},  pass to the limit as $\eps \to 0$  and consider \eqref{weakly-to-u} to get
\begin{equation} \label{eq: limit probl weak x}
\int_0^1 \Big( \frac{\partial^2 v}{\partial x^2} \frac{\partial^2 \psi}{\partial x^2} + \sigma \M(u) \frac{\partial^2 \psi}{\partial x^2} + \tau \frac{\partial v}{\partial x} \frac{\partial \psi}{\partial x} + v \psi   \Big) g(x)\, dx = \int_0^1 \M(f) \psi g(x)\, dx
\end{equation}
for all $\psi \in H^2_0(0,1)$.
Here,  the averaging operator $\M $ is defined from $L^2(R_1)$ to $L^2_g(0,1)$ by
\[
\M  h (x) = \frac{1}{ g(x)} \int_0^{ g(x)} h(x,y)\, dy\, ,
\]
for all $h\in  L^2(R_1)$ and for almost all $x \in (0,1)$.

From \eqref{eq: limit probl weak x} we deduce that
\[
\frac{1}{g} (v'' g)'' + \frac{\sigma}{g} (\M(u) g)'' - \frac{\tau}{g} (v' g)' + v = \M(f), \quad\quad \text{in (0,1)},
\]
where the equality is understood in the sense of distributions.\\
Coming back to \eqref{eq: weak formulation R1} we may also choose test functions $\varphi(x,y) = \eps^2 \zeta(x,y)$, where $\zeta \in H_{L_1}^2(R_1)$.  Using \eqref{ineq: decay y derivatives 1}, \eqref{ineq: decay y derivatives 2} and letting $\eps \to 0$ we deduce
\begin{equation*}
(1-\sigma) \int_{R_1} u \frac{\partial^2 \zeta}{ \partial y^2} + \sigma \int_{R_1}\Big( \frac{\partial^2 v}{\partial x^2} \frac{\partial^2 \zeta}{\partial y^2} + u \frac{\partial^2 \zeta}{\partial y^2}    \Big) = 0
\end{equation*}
which can be rewritten as
\begin{equation}
\label{eq: identity 2 order deriv}
\int_{R_1} \Big( u + \sigma \frac{\partial^2 v}{\partial x^2} \Big) \frac{\partial^2 \zeta}{\partial y^2} = 0
\end{equation}
for all $\zeta \in H_{L_1}^2(R_1)$. In particular this holds for all $\zeta \in C^\infty_c(R_1)$, hence there exists the second order derivative
\begin{equation} \label{eq: second derivative yy = 0}
\frac{\partial^2}{\partial y^2}\Big( u + \sigma \frac{\partial^2 v}{\partial x^2}  \Big) = 0.
\end{equation}
Hence, $u(x,y) + \sigma \frac{\partial^2 v}{\partial x^2} = \psi_1(x) + \psi_2(x) y$ for almost all $(x,y) \in R_1$ and for some functions $\psi_1, \psi_2 \in L^2(R_1)$, and then   \eqref{eq: identity 2 order deriv} can be written as
\begin{equation} \label{eq: limit probl weak y}
  \int_{R_1} (\psi_1(x)+y\psi_2(x)) \frac{\partial^2 \zeta}{\partial y^2} = 0
\end{equation}
Integrating twice by parts in $y$ in equation \eqref{eq: limit probl weak y} we deduce that
\begin{equation}
\label{eq: boundary identity}
- \int_{\partial R_1} \psi_2(x) \zeta n_y dS + \int_{\partial R_1} (\psi_1(x)+y\psi_2(x)) \frac{\partial \zeta}{\partial y} n_y dS = 0
\end{equation}
for all $\zeta \in H_{L_1}^2(R_1)$. We are going to choose now particular functions $\zeta$ in \eqref{eq: boundary identity}.  Consider first  $b=\frac{1}{2}\min_{x\in [0,1]} g(x)>0$ so that the rectangle $(0,1)\times (0,b)\subset R_1$ and consider a function $\eta=\eta(y)$ with $\eta\in C^\infty(0,b)$ such that $\eta(y)=1+\alpha y$ for $0<y<b/4$,  where $\alpha\in \R$ is a parameter, and $\eta(y)\equiv 0$ for $y\in(\frac{3}{4}b,b)$. If we define $\zeta(x,y)=\theta(x)\eta(y)$ for $(x,y)\in (0,1)\times (0,b)$  where $\theta\in C_c^\infty(0,1)$ and we extend this function $\zeta$  by 0 to all of $R_1$, then we can use $\zeta$ in \eqref{eq: boundary identity} in order to obtain
\[
\alpha \int_0^1 \psi_1(x)\theta(x) dx - \int_0^1\psi_2(x)\theta(x)dx=0
\]
for all $\alpha\in \R$ and all $\theta\in C_0^\infty(0,1)$. But this easily implies that $\psi_1\equiv \psi_2\equiv 0$. Thus, we obtain
\[
u(x,y) = u(x) = - \sigma \frac{\partial^2 v(x)}{\partial x^2}
\]
for almost all $(x,y) \in R_1$, i.e., $\frac{1}{\eps^2} \frac{\partial^2 \tilde{v}_\eps}{ \partial y^2} \rightharpoonup - \sigma  \frac{\partial^2 v(x)}{\partial x^2}$ in $L^2(R_1)$. Hence $v$ solves the following limit problem
\begin{equation}\label{ODE: auxiliary problem sigma2}
\begin{cases}
\frac{1 - \sigma^2}{g} (gv'')'' - \frac{\tau}{g}(gv')' + v = \M(f) , &\text{in $(0,1)$}\\
v(0)=v(1)=0,&\\
v'(0)=v'(1)=0,&
\end{cases}
\end{equation}
and then by regularity theory we deduce that $v \in H^4(0,1)$.

\subsection{Spectral convergence }
\label{sec: spectral convergence}
We aim at proving the spectral convergence of the eigenvalues and eigenfunctions of problem \eqref{PDE: R_eps} to the corresponding eigenvalues and eigenfunctions of the one dimensional problem \eqref{ODE: limit problem}. To do so we shall prove the compact convergence of the  associated  resolvent operators combined with the computations carried out in the previous section.  Note that  the domain $R_\eps$ varies with $\eps$, hence the corresponding Hilbert spaces vary as well.  To bypass this problem we will use the notion of $\E$-convergence of the resolvent operators in $L^2$.  We  recall the basic definitions and results.\\
Let $\Hi_\eps$, $\epsilon >0$, be a family of Hilbert spaces. We assume the existence of a family of linear operators $\E_\eps \in \LL(\Hi_0, \Hi_\eps)$, $\epsilon >0$,  such that
\begin{equation}
\label{def: basic property E_eps}
\norma{\E_\eps u_0}_{\Hi_\eps} \to \norma{u_0}_{\Hi_0},\ \ {\rm as}\ \eps\to 0,
\end{equation}
for all $u_0 \in \Hi_0$.

\begin{definition} Let $\Hi_\eps$ and $\E_\eps$ be as above.
\begin{enumerate}[label =(\roman*)]
\item Let $u_\eps\in   \Hi_\eps$, $\eps >0$.  We say that $u_\eps$ $\E$-converges to $u$ as $\eps \to 0$ if $\norma{u_\eps - \E_\eps u}_{\Hi_\eps} \to 0$ as $\eps \to 0$. We write $u_\eps \overset{E}{\longrightarrow} u$.
\item Let $ B_\eps \in \LL(\Hi_\eps)$,  $\eps >0$. We say that $B_\eps$ $\E\E$-converges to a linear operator $B_0 \in \LL(\Hi_0)$ if $B_\eps u_\eps \overset{E}{\longrightarrow} B_0 u$ whenever $u_\eps \overset{E}{\longrightarrow} u\in \Hi_0$. We write $B_\eps \overset{EE}{\longrightarrow} B_0$.
\item Let $ B_\eps \in \LL(\Hi_\eps)$, $  \eps >0$. We say that $B_\eps$ compactly converges to $B_0 \in \LL(\Hi_0)$ (and we write $B_\eps \overset{C}{\longrightarrow} B_0$) if the following two conditions are satisfied:
    \begin{enumerate}[label=(\alph*)]
    \item $B_\eps \overset{EE}{\longrightarrow} B_0$ as $\eps \to 0$;
    \item for any family $u_\eps \in \Hi_{\eps}$, $\eps>0$, such that $\norma{u_\eps}_{\Hi_\eps}=1$ for all $\eps \in (0,1)$, there exists a subsequence $B_{\eps_k}u_{\eps_k}$ of $B_\eps u_\eps$ and $\bar{u} \in \Hi_0$ such that $B_{\eps_k}u_{\eps_k} \overset{E}{\longrightarrow} \bar{u}$ as $k \to \infty$.
    \end{enumerate}
\end{enumerate}
\end{definition}

For any $\eps \geq 0$, let $A_\eps$ be a (densely defined) closed, nonnegative differential operator on $\Hi_\eps$ with domain $\D(A_\eps) \subset \Hi_\eps$.  We assume for simplicity that $0$ does not belong to the spectrum of $A_{\eps}$ and that
\[
\textup{(H1): $A_\eps$ has compact resolvent $B_\eps := A_\eps^{-1}$ for any $\eps \in [0,1)$,}
\]
and
\[
\textup{(H2): $B_\eps \overset{C}{\longrightarrow} B_0 $, as $\eps \to 0$.}
\]
Given an eigenvalue $\lambda$ of $A_0$ we consider the generalized eigenspace $S(\lambda, A_0) := Q(\lambda, A_0)\Hi_0$, where
\[
Q(\lambda, A_0) = \frac{1}{2\pi i} \int_{|\xi - \lambda| = \delta} (\xi \mathbb{I} - A_0)^{-1}\, d\xi
\]
and $\delta > 0$ is such that the disk $\{\xi \in \mathbb{C} : |\xi - \lambda| \leq \delta \}$ does not contain any eigenvalue except for $\lambda$. In a similar way, if (H1),(H2) hold, then we can define   $S(\lambda, A_{\eps}) := Q(\lambda, A_{\eps})\Hi_{\eps}$, where
\[
Q(\lambda, A_\eps) = \frac{1}{2\pi i} \int_{|\xi - \lambda| = \delta} (\xi \mathbb{I} - A_\eps)^{-1}\, d\xi.
\]
This definition makes sense because for $\eps$ small enough $(\xi \mathbb{I} - A_\eps)$ is invertible for all $\xi$ such that $|\xi - \lambda| = \delta$, see \cite[Lemma 4.9]{ACJdE}. Then the following theorem holds.

\begin{theorem}
\label{thm: E conv -> spectral conv}
Let $A_\eps$, $A_0$ be operators as above satisfying  conditions (H1), (H2). Then the operators $A_{\eps}$ {\rm are spectrally convergent} to $A_0$ as $\eps\to 0$, i.e., the following statements hold:
\begin{enumerate}[label=(\roman*)]
\item If $\lambda_0$ is an eigenvalue of $A_0$, then there exists a sequence of eigenvalues $\lambda_\eps$ of $A_\eps$ such that $\lambda_\eps \to \lambda_0$ as $\eps \to 0$. Conversely, if $\lambda_\eps$ is an eigenvalue of $A_\eps$ for all $\eps >0$, and $\lambda_\eps \to \lambda_0$, then $\lambda_0$ is an eigenvalue of $A_0$.
\item There exists $\eps_0 > 0$ such that the dimension of the generalized eigenspace $S(\lambda_0, A_\eps)$ equals the dimension of $S(\lambda_0, A_0)$, for any eigenvalue $\lambda_0$ of $A_0$, for any $\eps \in [0,\eps_0)$.
\item If $\varphi_0 \in S(\lambda_0, A_0)$ then for any $\eps >0$ there exists  $\varphi_\eps \in S(\lambda_0, A_\eps)$ such that  $\varphi_\eps \overset{E}{\longrightarrow} \varphi_0$ as $\eps \to 0$.
\item If $\varphi_\eps \in S(\lambda_0, A_\eps)$ satisfies $\norma{\varphi_\eps}_{\Hi_\eps} = 1$ for all $\eps >0$, then  $\varphi_\eps $, $\eps >0$,  has an $\E$-convergent subsequence whose limit is in $S(\lambda_0, A_0)$.
\end{enumerate}
\end{theorem}
\begin{proof}
See \cite[Theorem 4.10]{ACL}.
\end{proof}

We now apply Theorem \ref{thm: E conv -> spectral conv} to problem \eqref{PDE: R_eps}.  To do so, we consider the following Hilbert spaces
\[
\Hi_\eps = L^2(R_\eps; \eps^{-1}dxdy),\ \ {\rm and}\ \  \Hi_0 = L^2_g(0,1),
\]
and we denote by   $\E_\eps$   the extension operator from $L^2_g(0,1)$ to  $L^2(R_\eps; \eps^{-1}dxdy)$, defined by
\begin{equation}
\label{def: extension}
(\E_\eps v)(x,y) = v(x),
\end{equation}
for all $v \in L^2_g(0,1)$, for almost all $(x,y) \in R_\eps$. Clearly $\norma{E_\eps u_0}_{(R_\eps; \eps^{-1}dxdy)} = \norma{u_0}_{L^2_g(0,1)}$, hence $\E_\eps$ trivially satisfies property \eqref{def: basic property E_eps}.

We consider the operators $A_{\eps}=(\Delta^2 - \tau \Delta +I)_{ L_{\eps }}$,  $A_0=(\Delta^2 - \tau \Delta +I)_{D }$  on $\Hi_{\eps}$ and $\Hi_0$ respectively, associated with the eigenvalue problems \eqref{PDE: R_eps} and  \eqref{ODE: limit problem}, respectively. Namely, $(\Delta^2 - \tau \Delta +I)_{ L_{\eps }}$ is the operator $\Delta^2 - \tau \Delta +I$ on $R_{\epsilon}$ subject to Dirichlet boundary
conditions on $L_{\eps}$ and Neumann  boundary conditions on $\partial R_{\eps}\setminus L_{\epsilon }$ as described in  \eqref{PDE: R_eps}. Similarly, $(\Delta^2 - \tau \Delta +I)_{D }$
is the operator $\Delta^2 - \tau \Delta +I$ on $(0,1)$ subject to Dirichlet boundary conditions as described in \eqref{ODE: limit problem}.

Then we can prove the following
\begin{theorem}\label{spectfin} The operators  $(\Delta^2 - \tau \Delta +I)_{ L_{\eps }}$  spectrally converge to\\  $(\Delta^2 - \tau \Delta +I)_{D }$ as $\eps \to 0$,  in the sense of Theorem~\ref{thm: E conv -> spectral conv}.
\end{theorem}

\begin{proof}
In view of Theorem \ref{thm: E conv -> spectral conv},  it is sufficient to prove the following two facts:
\begin{enumerate}[label=(\arabic*)]
\item if  $f_\eps \in L^2(R_\eps; \eps^{-1}dxdy)$ is  such that $\eps^{-1/2}\norma{f_\eps}_{L^2(R_\eps)} = 1$ for any $\eps >0$, and  $v_\eps$ is the corresponding solutions of Problem \eqref{PDE: R_eps f_eps}, then there exists a subsequence $\eps_k \to 0$ as $k \to \infty$ and $\bar{v} \in L^2_g(0,1)$ such that $v_{\eps_k}$ $\E$-converge to $\bar{v}$ as $k \to \infty$.
\item if  $f_\eps \in L^2(R_\eps; \eps^{-1}dxdy)$ and $f_\eps  \overset{E}{\longrightarrow} f$ as $\eps \to 0$, then the corresponding solutions $v_\eps$ of Problem \eqref{PDE: R_eps f_eps} $\E$-converge to the solution of Problem \eqref{ODE: auxiliary problem sigma2} with datum $f$.
\end{enumerate}
Note that (1) follows immediately from the computations in Section \ref{subsection: finding limit prb}. Indeed, if $f_\eps \in L^2(R_\eps; \eps^{-1}dxdy)$ is as in (1), up to a subsequence, $\tilde{f}_\eps \rightharpoonup f$ in $L^2(R_1)$, which implies that $\tilde{v}_\eps \rightharpoonup v_0\in H_0^2(0,1)$ in $H^2(R_1)$, where $v_0$ is the solution of Problem \eqref{ODE: auxiliary problem sigma2}. This implies that $\norma{v_\eps - \E v_0}_{L^2(R_\eps; \eps^{-1}dxdy)} \to 0$, hence (1) is proved.\\
In order to show (2) we take a sequence of functions $f_\eps \in L^2(R_\eps; \eps^{-1}dxdy)$ and $f\in L^2_g(0,1)$ such that $\eps^{-1/2}\norma{f_\eps - \E_\eps f}_{L^2(R_\eps)} \to 0$ as $\eps \to 0$. After a change of variable, this is equivalent to $\norma{\tilde{f}_\eps - \E f}_{L^2(R_1)} \to 0$ as $\eps \to 0$. Arguing as in Section \ref{subsection: finding limit prb}, one show that the $\tilde{v}_\eps \rightharpoonup v \in L^2_g(0,1)$ in $H^2(R_1)$ and that $v$ solves problem \eqref{ODE: auxiliary problem sigma2}. Hence $\norma{\tilde{v}_\eps - \E v}_{L^2(R_1)} \to 0$ as $\eps \to 0$, or equivalently, $\norma{v_\eps - \E_\eps v}_{L^2(R_\eps; \eps^{-1}dxdy)} \to 0$ as $\eps \to 0$, proving (2).

\end{proof}

%---------------------------------------END OF "Asymptotic analysis on the thin domain" SECTION--------------------------------------

\section{Conclusion}\label{conclusionsec}

Recall that the eigenpairs of problems \eqref{PDE: main problem_eigenvalues}, \eqref{PDE: Omega} are denoted by $(\lambda_n(\Omega_\eps), \varphi_n^\eps)$, $(\omega_n, \varphi_n^\Omega)_{n\geq1}$ respectively, where the two families of eigenfunctions $\varphi_n^\eps$, $\varphi_n^\Omega$  are complete orthonormal bases of the spaces $L^2(\Omega_{\epsilon})$, $L^2(\Omega )$, respectively.    Denote now by $(h_n, \theta_n)_{n\geq 1}$ the eigenpairs of problem $\eqref{ODE: limit problem}$
where  the eigenfunctions $h_n$ define an orthonormal basis of the space $L^2_g(0,1)$.
In the spirit of the definition of $\lambda_n^{\eps} $ given in Section 2, we  set now  $(\lambda_n^0)_{n\geq 1} = (\omega_k)_{k \geq 1} \cup (\theta_l )_{l\geq 1}$, where it is understood that the eigenvalues are arranged in increasing order and repeated according to their multiplicity.  For each $\lambda_n^0$ we define the function  $\phi_n^0 \in H^2(\Omega) \oplus H^2(R_\eps)$ in the following way:
\begin{equation*}
\phi^0_n =
\begin{cases}
    \varphi_k^\Omega, &\text{in $\Omega$}\\
    0, &\text{in $R_\eps$},
\end{cases}
\end{equation*}
if $\lambda_n^0 = \omega_k$, for some $k \in \N$; otherwise
\begin{equation*}
\phi^0_n=
\begin{cases}
    0, &\text{in $\Omega$}, \\
    \eps^{-1/2}\E_\eps h_l, &\text{in $R_\eps$}
\end{cases}
\end{equation*}
if $\lambda_n^\eps = \theta_l$, for some $l \in \N$ (here we agree to order the eigenvalues and the  eigenfunctions following the same rule used in the definition of $\lambda_n^{\eps}$ and $\phi_n^{\epsilon }$ in Section 2).

Finally,  if $x>0$ divides the spectrum $\lambda_n(\Omega_{\epsilon})$ for all $\eps >0 $ sufficiently small  (see the end of Section 2)   and $N(x)$ is the number of eigenvalues with $\lambda_n(\Omega_{\epsilon})\le x$ (counting their multiplicity),  we  define the projector $P_{x}^0$ from $L^2(\Omega_\eps)$ onto the linear span $[\phi_1^{0}, \dots, \phi_{N(x)}^{0}]$  by setting
\[
P_{x}^0 u = \sum_{i=1}^{N(x)} (u,\phi_i^0)_{L^2(\Omega_\eps)} \phi_i^0
\]
for all $u\in L^2(\Omega_\eps)$. (Note that choosing $x$ independent of $\eps$ is possible by the limiting behaviour of the eigenvalues.) Then, using Theorems \ref{thm: eigenvalues decomposition} and  \ref{spectfin}  we deduce the following.

\begin{theorem} \label{lastthm}
Let $\Omega_\eps$, $\eps>0$, be a family of  dumbbell domains satisfying the H-Condition. Then the following statements hold:
\begin{enumerate}[label =(\roman*)]
\item   $\lim_{\eps \to 0}\, \abs{\lambda_n(\Omega_\eps) - \lambda_n^0} = 0$,  for all  $n\in \N $.
\item   For any $x$ dividing the spectrum,
 $\lim_{\eps \to 0}\, \norma{\varphi^\eps_{n} - P^0_{x} \varphi^\eps_{n}}_{H^2(\Omega) \oplus L^2(R_\eps)} = 0$, for all $n = 1,\dots, N(x)$.
\end{enumerate}
\end{theorem}
\begin{proof}
The convergence of the eigenvalues follows directly by  Theorems \ref{thm: eigenvalues decomposition} and  \ref{spectfin}.  Indeed, by Theorem \ref{thm: eigenvalues decomposition} we know that $|\lambda_n(\Omega_\eps) - \lambda_n^\eps| \to 0$ as $\eps \to 0$. If $\lambda_n^\eps = \omega_k$ for some $k\in \N$, for all sufficiently small $\eps$, then we are done; otherwise, if $\lambda_n^\eps = \theta_l^\eps$ for some $l\in \N$, definitely in $\eps$, by Theorem \ref{spectfin} we deduce that $\theta_l^\eps \to \theta_l$ as $\eps \to 0$, hence $|\lambda_n(\Omega_\eps) - \theta_l| \leq |\lambda_n(\Omega_\eps) - \theta_l^\eps| + |\theta_l^\eps - \theta_l| \to 0 $ as $\eps \to 0$.\\
Consider now the convergence of the eigenfunctions. By Theorems~\ref{thm: E conv -> spectral conv},~\ref{spectfin} it follows that for any $\epsilon >0$ there exists an orthonormal  sequence of generalized eigenfunctions $\delta_j^{\eps }$ in $L^2(R_{\eps}, \eps^{-1}dxdy)$ associated with the eigenvalues
$\theta_j^{\eps}$ of problem \eqref{PDE: R_eps} such that for every $j\in \N$
\begin{equation}\label{lastthm1}
\norma{\delta^\eps_j - \E_{\eps} h_j}_{L^2(R_\eps,  \eps^{-1}dxdy )} \to 0,
\end{equation}
as $\eps \to 0$.  Recall that  a generalized eigenfunction is an element of a generalized eigenspace, see Section~\ref{sec: spectral convergence}.  We set
$
\gamma_j^\eps =\eps^{-1/2}\delta^\eps_j
$
and we note that $\gamma_j^{\eps}$ is a sequence of generalized eigenfunctions of Problem \eqref{PDE: R_eps} which is orthonormal in $L^2(R_{\eps})$, as required in Theorem~\ref{thm: eigenvalues decomposition}. Thus by Theorem~\ref{thm: eigenvalues decomposition} $(ii)$,  we deduce that
\small{
\begin{equation*}
\begin{split}
&\norma*{\varphi_n^{\eps} - \sum^{N(x) }_{i=1} (\varphi_n^\eps, \eps^{-1/2}\E_\eps h_i)_{L^2(R_\eps)} \eps^{-1/2} \E_\eps h_i}_{L^2(R_\eps)} \leq \norma*{\varphi_n^{\eps} - \sum^{N(x)}_{i=1} (\varphi_n^\eps, \gamma_i^\eps)_{L^2(R_\eps)} \gamma_i^\eps }_{L^2(R_\eps)}\\
&+ \norma*{\sum^{N(x)}_{i=1} (\varphi_n^\eps, \gamma_i^\eps)_{L^2(R_\eps)} \gamma_i^\eps - \sum^{N(x)}_{i=1} (\varphi_n^\eps, \eps^{-1/2}\E_\eps h_i)_{L^2(R_\eps)} \eps^{-1/2}\E_\eps h_i  }_{L^2(R_\eps)}\\
&\leq o(1) + \norma*{\sum^{N(x)}_{i=1} (\varphi_n^\eps, \eps^{-1/2}\E_\eps h_i)_{L^2(R_\eps)} ( \gamma_i^\eps -\eps^{-1/2}\E_\eps h_i)  }_{L^2(R_\eps)} + \norma*{\sum^{N(x)}_{i=1} (\varphi_n^\eps, \gamma_i^\eps - \eps^{-1/2}\E_\eps h_i)_{L^2(R_\eps)} \gamma_i^\eps}_{L^2(R_\eps)}\\
%&\leq o(1) + C \sum^{N(x)}_{i=1} \norma{\gamma_i^\eps - \eps^{-1/2}\E_\eps h_i}_{L^2(R_\eps)} =o(1) + C \sum^{N(x)}_{i=1} \norma{\delta _i^\eps - \E_\eps h_i}_{L^2(R_\eps ,\eps^{-1}dxdy   )} .\\
\end{split}
\end{equation*}}
\normalsize
Hence, 
\begin{equation}\label{lastthm2}
\begin{split}
&\norma*{\varphi_n^{\eps} - \sum^{N(x) }_{i=1} (\varphi_n^\eps, \eps^{-1/2}\E_\eps h_i)_{L^2(R_\eps)} \eps^{-1/2} \E_\eps h_i}_{L^2(R_\eps)}\\ 
&\leq o(1) + C \sum^{N(x)}_{i=1} \norma{\gamma_i^\eps - \eps^{-1/2}\E_\eps h_i}_{L^2(R_\eps)} =o(1) + C \sum^{N(x)}_{i=1} \norma{\delta _i^\eps - \E_\eps h_i}_{L^2(R_\eps ,\eps^{-1}dxdy   )}.
\end{split}
\end{equation}
Since the right-hand side of the last inequality in  \eqref{lastthm2} goes to zero as $\eps \to 0$ by (\ref{lastthm1}), we conclude that $\lim_{\eps \to 0}\, \norma{\varphi^\eps_{n} - P^0_{x} \varphi^\eps_{n}}_{L^2(R_\eps)} = 0$. Finally, the fact that $\lim_{\eps \to 0}\, \norma{\varphi^\eps_{n} - P^0_{x} \varphi^\eps_{n}}_{H^2(\Omega)} = 0$ follows directly from Theorem \ref{thm: eigenvalues decomposition}.
\end{proof}

%-----------------------------------END OF "Conclusions" SECTION-------------------------------------------------------------------

% ------------------------------------------------------------------------

\subsection*{Acknowledgment}
The first author is partially supported by grants MTM2012-31298, MTM2016-75465, ICMAT Severo Ochoa project SEV-2015-0554, MINECO, Spain and Grupo de Investigaci\'on CADEDIF, UCM. The third author acknowledges financial support from the INDAM - GNAMPA project 2016 ``Equazioni differenziali con applicazioni alla meccanica". The second and third authors are also members of the Gruppo Nazionale per l'Analisi Matematica, la Probabilit\`{a} e le loro Applicazioni (GNAMPA) of the Istituto Nazionale di Alta Matematica (INdAM).\\
The second and the third authors are very thankful to the Departamento de Matem\'atica Aplicada of the Universidad Complutense de Madrid for the warm hospitality received on the occasion of their visits. The authors are thankful to an anonymous referee for pointing out a number of items in the reference list.

% ------------------------------------------------------------------------
\end{document}